\documentclass[11 pt]{amsart}

\usepackage{hyperref}
\usepackage{amssymb, amsmath, amsthm, amsfonts}
\usepackage{verbatim}
\usepackage{bbm}
\usepackage{enumerate}
\usepackage{dsfont}
\usepackage{upgreek}
\usepackage[mathscr]{eucal}
\usepackage{tikz}
\usepackage[normalem]{ulem}

\usepackage[backgroundcolor=black!5,linecolor=black]{todonotes}
\usepackage{comment}

\usepackage[explicit,indentafter]{titlesec}
\titleformat{\section}{\centering\normalfont\scshape}{\thesection.}{.5em}{#1}
\titleformat{\subsection}[runin]{\normalfont\itshape}{\textnormal{\thesubsection.}}{.5em}{#1.}
\titleformat{\subsubsection}[runin]{\normalfont\itshape}{\thesubsubsection.}{.5em}{#1.}
\titlespacing{\section}{0em}{1em}{0.5em}
\titlespacing{\subsection}{0em}{.5em}{0.5em}

\usepackage{stackengine}

\newcommand{\epo}{\epsilon_\circ}

\newcommand{\ox}{{\overline x}}
\newcommand{\oy}{{\overline y}}

\newcommand{\ux}{{\underline x}}
\newcommand{\uy}{{\underline y}}

\newcommand{\vr}{{\varrho}}

\usepackage{color}
\definecolor{gray}{gray}{0.5}

\newcommand{\cmt}[1]{}

\newcommand{\vertiii}[1]{{\left\vert\kern-0.25ex\left\vert\kern-0.25ex\left\vert #1 
    \right\vert\kern-0.25ex\right\vert\kern-0.25ex\right\vert}}
    
    \newcommand{\detail}[1]{}

	\newcommand{\drawQbg}{
		\fill (Q1) node [left] {$Q_1$} circle [radius=.02em];
		\fill (Q2) node [above left] {$Q_{2}$} circle [radius=\ptsize];
		\fill (Q3) node [right] {$Q_{3}$} circle [radius=\ptsize];
		\fill (Q4g) node [below right] {$Q_{4}$} circle [radius=\ptsize];
		\fill [color=\QbgFillColor,opacity=\QbgFillOpacity] (Q1) -- (Q2) -- (Q3) -- (Q4g) -- cycle;
		\draw [opacity=\QbgLineOpacity] (Q1) -- (Q2) -- (Q3);
		\draw [opacity=\QbgLineOpacity] (Q4g) -- (Q1);
		\if\QbgDrawCritSeg1
			\draw [style=\QbgCritSegStyle,opacity=\QbgCritSegOpacity] (Q3) -- (Q4g);
		\fi
	}

\newcommand{\hor}{\mathrm{hor}}
\newcommand{\vt}{\mathrm{vert}}

\def\lc{\lesssim}
\def\gc{\gtrsim}

\def\sR{\mathscr{R}}
\def\sV{\mathscr{V}}

\def\eps{\varepsilon}
\def\bbone{{\mathbbm 1}}

\newcommand{\sK}{\mathscr K}

\newcommand{\wt}{\widetilde}

\newcommand{\floor}[1]{\lfloor #1 \rfloor }

\newcommand{\Be}{\begin{equation}}
\newcommand{\Ee}{\end{equation}}

\newcommand{\Bm}{\begin{multline}}
\newcommand{\Em}{\end{multline}}

\def\intslash{\rlap{\kern  .32em $\mspace {.5mu}\backslash$ }\int}
\def\qsl{{\rlap{\kern  .32em $\mspace {.5mu}\backslash$ }\int_{Q_x}}}

\def\lc{\lesssim}
\def\gc{\gtrsim}

\def\floor#1{{\lfloor #1 \rfloor }}
\def\emph#1{{\it #1 }}
\def\diam{{\text{\rm  diam}}}

\def\ga{\gamma}

\def\cf{{\it cf}}

\def\meas{{\mathrm{meas}}}

\def\supp{{\mathrm{supp}}}

\def\inn#1#2{\langle#1,#2\rangle}
\def\biginn#1#2{\big\langle#1,#2\big\rangle}

\def\ga{\gamma}             
\def\eps{\varepsilon}

\def\ka{\kappa}
\def\la{\lambda}             

\def\om{\omega}              \def\Om{\Omega}
\def\vth{\vartheta}

\def\fA{{\mathfrak {A}}}

\def\fM{{\mathfrak {M}}}

\def\fS{{\mathfrak {S}}}

\def\fc{{\mathfrak {c}}}

\def\fe{{\mathfrak {e}}}

\def\fg{{\mathfrak {g}}}

\def\fr{{\mathfrak {r}}}
\def\fs{{\mathfrak {s}}}

\def\fw{{\mathfrak {w}}}

\def\fz{{\mathfrak {z}}}

\def\bbH{{\mathbb {H}}}

\def\bbN{{\mathbb {N}}}

\def\bbQ{{\mathbb {Q}}}
\def\bbR{{\mathbb {R}}}

\def\bbZ{{\mathbb {Z}}}

\def\cA{{\mathcal {A}}}

\def\cE{{\mathcal {E}}}
\def\cF{{\mathcal {F}}}
\def\cG{{\mathcal {G}}}

\def\cJ{{\mathcal {J}}}
\def\cK{{\mathcal {K}}}
\def\cL{{\mathcal {L}}}

\def\cR{{\mathcal {R}}}

\def\cT{{\mathcal {T}}}
\def\cU{{\mathcal {U}}}

\def\emph#1{{\it #1}}
\def\textbf#1{{\bf #1}}

\def\beq{\begin{equation}}
\def\endeq{\end{equation}}

\def\bs{\begin{split}}
\def\es{\end{split}}

\theoremstyle{plain}

\newtheorem{thm}{Theorem}[section]
\newtheorem{prop}[thm]{Proposition}

\newtheorem{lem}[thm]{Lemma}
\newtheorem{cor}[thm]{Corollary}

\newtheorem*{thm*}{Theorem}
\newtheorem*{conj*}{Conjecture}
\newtheorem*{openproblem*}{Open Problem}

\theoremstyle{remark}
\newtheorem{rem}[thm]{Remark}

\newtheorem*{remarka}{Remark}

\numberwithin{equation}{section}

\definecolor{jrcol}{rgb}{0,0,1.} 
\definecolor{ascol}{rgb}{1.,0,0} 

\def\R{\mathbb{R}}

\begin{document}
\title
[Spherical maximal functions on two step nilpotent groups]
{Spherical maximal functions on \\two step nilpotent Lie groups}
\author[Jaehyeon Ryu and Andreas Seeger]{Jaehyeon Ryu \ \ \ \ \ Andreas Seeger} 

\address{Jaehyeon Ryu: School of Mathematics, Korea Institute for Advanced Study, Seoul 02455, Republic of
Korea  \& Department of Mathematics, University of Wisconsin, 480 Lincoln Drive, Madison, WI, 53706, USA.}

\email{jhryu@kias.re.kr}

\address{Andreas Seeger: Department of Mathematics, University of Wisconsin, 480 Lincoln Drive, Madison, WI, 53706, USA.}
\email{seeger@math.wisc.edu}

\maketitle 

\begin{abstract}
Consider $\mathbb R^d\times \mathbb R^m$ with the group structure of a two-step nilpotent Lie group and natural parabolic dilations.  
The maximal function  originally 
introduced by Nevo and Thangavelu  in the setting of the Heisenberg group deals with   noncommutative convolutions associated to measures on  spheres or generalized spheres in $\mathbb R^d$.  We drop the nondegeneracy assumptions  in the known results on M\'etivier groups and prove the sharp $L^p$ boundedness result for all  two step nilpotent Lie groups with $d\ge 3$. 
\end{abstract}

\section{Introduction}
We consider the  problem of bounding maximal operators for averages over  spheres with higher codimension on a two-step nilpotent Lie group $G$ which was introduced for the special case of the Heisenberg group  by Nevo and Thangavelu \cite{NevoThangavelu1997}. The setup is as follows: The Lie algebra 
splits as a direct sum 
in two subspaces referred to as the horizontal and the vertical part,
$\fg=\fw_{\hor}\oplus \fw_{\vt}$,
where $\dim\fw_\hor=d$, $\dim\fw_\vt=m$,  and  $\fw_\vt\subseteq\fz(\fg)$, with $\fz(\fg)$ the center of the Lie algebra. We use the natural parabolic dilation structure on $\fw_\hor\oplus\fw_\vt$, and define for $\underline X\in \fw_\hor$, $\overline X\in \fw_\vt$, $\delta_t(\underline X, \overline X)= (t\underline X,t^2 \overline X)$.
 Using exponential coordinates  on the group we identify $G$ with $\fw_\hor\oplus\fw_\vt\equiv\bbR^d\oplus \bbR^{m}$. 
 With  $x=(\underline x,\overline x)\in \bbR^d\times \bbR^m$  the group law then  becomes
\Be\label{eq:grouplaw} (\underline x,\overline x)\cdot (\underline y, \overline y)= (\underline x+\underline y, \overline x+\overline y + \underline x^\intercal \vec J \underline y)
\Ee 
where $\underline x^\intercal \vec J \underline y:= 
\sum_{i=1}^m  \overline e_i \underline x^\intercal J_i \underline y$,  
$\{\overline e_i\}_{i=1}^m $ is the standard basis of unit vectors in $\bbR^m$
and $J_1,\dots, J_m$ are $d\times d$ skew-symmetric matrices. 
The above dilations on $\fg$ induce   automorphisms $\delta_t: (\ux,\ox)\mapsto (t\ux, t^2\ox)$ on the group.  We will study averaging operators which will be convolution operators; the noncommutative convolution for functions $f*K(x)=\int f(y)K(y^{-1}\cdot x) dy$ is then given in the form 
\begin{align} \label{eq:conv} f*K(\ux,\ox)&= \int f(\uy, \oy) K(\ux-\uy, \ox-\oy+\ux^\intercal\vec J \uy) \, dy \\ 
\notag&= \int f(\underline  x-\underline  w, \overline x-\overline w-\underline x^\intercal \vec J\underline w) \, K(\underline w, \overline w) dw.
\end{align}

Let $\Om$ be a bounded open convex domain  in $\fw_{\mathrm{hor}}\equiv \bbR^d\times\{0\}$ containing the  origin  and assume throughout the paper that the boundary  $\Sigma\equiv \partial\Om$ is smooth 
with {\it nonvanishing Gaussian curvature.} In most previous papers one takes for $\Om$ the 
unit ball in $\bbR^d\times\{0\}$.  Let 
 $\mu$ be the normalized surface measure on $\Sigma$. 
For $t>0$ the dilate $\mu_t$ is defined by $\inn{f}{\mu_t} =\int f(tx,0)d\mu$. For Schwartz functions $f$ on $\bbR^{d+m}$  the averages over dilated spheres are then given by the convolutions 
\Be\label{eq:Adef} 
A f(x,t)= f\ast \mu_t(x) = \int_{\Sigma} f\big(\ux-t\om, \ox - t  \,\ux^\intercal \vec J \om\big) d\mu(\om).
\Ee
The analogue of the Nevo--Thangavelu  maximal operator is defined (a priori  for Schwartz  functions) by
\Be\label{fMdef}
\fM f(x) := \sup_{t>0} \big|f\ast \mu_t(x)\big|.
\Ee

The objective is to establish  an $L^p(\bbR^{d+m})\to L^p(\bbR^{d+m})$ bound for $\fM$, in  an optimal range of $p$.  Taking $\Sigma=S^{d-1}$ a  partial boundedness result for $p>\frac{d-1}{d-2}$ was first obtained by Nevo and Thangavelu on the Heisenberg groups $\bbH^n$,  for $2n\equiv d\ge 4$; here $m=1$ and $J=J_1$ is an invertible symplectic   matrix.  The optimal result on $L^p$ boundedness on the Heisenberg group $\bbH^n$, for $n\ge 2$, namely that $\fM$ is bounded on $L^p(\bbR^{d+1})$ for $p>\frac{d}{d-1} $ was  obtained  by M\"uller and the second author \cite{MuellerSeeger2004} and  independently by Narayanan and Thangavelu \cite{NarayananThangavelu2004}. The  paper \cite{MuellerSeeger2004} also establishes  this result in   the more general setting of M\'etivier groups,  that is, under the  nondegeneracy condition that for all $\theta\in \bbR^m\setminus \{0\}$ the matrices $\sum_{i=1}^m \theta_iJ_i$  are invertible. Regarding the case $n=1$ it is not currently known whether $\fM$ is bounded on  $L^p(\bbH^1)$ for any $p<\infty$ (see however results restricted to Heisenberg radial functions in \cite{BeltranGuoHickmanSeeger} and \cite{LeeLee}). 

The purpose of this paper is to examine the behavior of the maximal function  on general two-step nilpotent Lie groups with $d\ge 3$, i.e. when the nondegeneracy condition on M\'etivier groups fails. A trivial special case occurs  when all the matrices $J_i$ are zero; in this  case one immediately  obtains  the same $L^p$ boundedness result for $p>\frac{d}{d-1}$,  $d\ge 3$  by applying Stein's result \cite{SteinPNAS1976} (or Bourgain's result \cite{Bourgain-sphmax} when $d=2$) in the horizontal hyper-planes  and then integrating in $\fw_{\mathrm{vert}}$.  The two extreme cases of Euclidean and M\'etivier groups suggest that $L^p$ boundedness for $p>\frac{d}{d-1} $ should hold independently of the choice  of the matrices $J_i$. 
However the  intermediate cases are  harder, and neither the slicing argument nor the arguments  in \cite{MuellerSeeger2004, NarayananThangavelu2004, KimJoonil, RoosSeegerSrivastava-imrn} for the Heisenberg and M\'etivier cases seem to apply; this was posed as a problem in \cite{MuellerSeeger2004}. In particular there seems to be  no   regularity theorem on Fourier integral operators which covers the averages in this general case.  The   special case $m=1$ was recently considered by 
Liu and Yan \cite{LiuYan} who obtain $L^p$ boundedness of $\fM$ in the partial range 
$p> \frac{d-1}{d-2}$ and $d\ge 4$.  Here we prove the optimal result in the range 
$p>\frac{d}{d-1}$,  for all two-step nilpotent Lie groups  with $d\ge 3$.

\begin{thm}\label{thm:main}
    Let $d\ge 3$,  let $G$ be a general two step  nilpotent Lie group of dimension $d+m$, with group law \eqref{eq:grouplaw}. Let $\frac{d}{d-1}<p<\infty$. Then for $f\in L^p(G)$ and almost every $x\in G$  the functions $t\to Af(x,t)$ are continuous and the maximal  operator   $\fM$ 
     extends to a bounded operator on $L^p(G)$.
\end{thm}
    

By a standard argument Theorem \ref{thm:main} implies
\begin{cor} Let $f\in L^p(G)$, $p>\frac{d}{d-1}$. Then $\lim_{t\to 0} Af(x,t)= f(x)$ almost everywhere.\end{cor} 

To show in Theorem \ref{thm:main}  the continuity in $t$,  for a.e. $x\in G$,  we shall  prove a stronger inequality involving the standard Besov space  $B^{1/p}_{p,1}(\bbR)$  which is embedded in the space of bounded continuous functions. Namely,  for $u\in C^\infty_c((\tfrac 14,4))$ and $\la>0$  the functions $s\to u(s)A_{\la s} f(x)$ are in $B^{1/p}_{p,1}$;  this implies that  $\sup_{t>0}|A_t f(x)|=\sup_{t\in \bbQ} |A_t f(x)|$ almost everywhere and establishes  $\fM f$ as a well defined  measurable function, for every $f\in L^p$. In fact one gets the following global result which implies Theorem \ref{thm:main}.
\begin{thm}\label{thm:stronger} Let $d\ge 3$, $\frac{d}{d-1}<p<\infty$. Then for $u\in C^\infty_c((\frac 14,4)) $ 
\Be\label{eq:stronger}  \Big(\int_G \big[\sup_{n\in \bbZ} \|u(\cdot)  Af(x, 2^n \cdot)\|_{B^{1/p}_{p,1}(\bbR) } \big]^p dx\Big)^{1/p} \lc \|f\|_p.\Ee
\end{thm}

\begin{rem}\label{rem:higher-s-derv}
    The smoothness parameter for  the Besov space in the $s$-variable can   be increased in Theorem \ref{thm:stronger}. In fact we can replace $B^{1/p}_{p,1}$  with $B^{\beta+1/p}_{p,1} $ where  $\beta<d-1+\frac dp$ if $\frac{d}{d-1}< p\le 2$ and $\beta<\frac{d-2}{p} $ if $2\le p<\infty$, see also a relevant discussion in \S\ref{sec:global}.
    Such improvements will not yield additional insights on the maximal operator. 
\end{rem}


It is also of interest to consider  a local variant for which we get a restricted weak type inequality at the endpoint 
 $p=\frac{d}{d-1}$:
\begin{thm}\label{thm:rwt} Let $d\ge 3$, $p\ge\frac{d}{d-1}$ and let $I\subset (0,\infty)$  be a compact interval. Then $A$ maps $L^{p,1}(G)$  to $L^{p,\infty}(G; L^\infty(I))$.
\end{thm}

The optimality of the $p$-range in the above maximal function theorems is shown by modifying an example of Stein \cite{SteinPNAS1976}, see also the discussion in \cite{LiuYan}.

\subsection*{\it Outline of the paper and methodology} In \S\ref{sec:prel} we reduce matters to the case where the matrices $J_1,\dots, J_m$ are linearly independent.  In \S\ref{sec:dyadic} we set up 
standard dyadic frequency decompositions of the underlying spherical measures 
and formulate the main Proposition \ref{prop:mainAkl} to be proved for the boundedness of the local maximal operator (with dilation parameters in a compact interval). The arguments to extend to the global maximal operator 
(and the slightly stronger Theorem \ref{thm:stronger}) are modifications from those in  \cite{MuellerSeeger2004}; this is taken up  in \S\ref{sec:global}. The main $L^2$ estimates are discussed in \S\ref{sec:propprel}; here we first recast Proposition \ref{prop:mainAkl} in a convenient form using Fourier integral operators, and then reduce matters to the problem of getting uniform estimates for a family of oscillatory integral operators acting on functions in $\bbR^d$. The main $L^2$ estimates for this family are stated in Proposition \ref{prop:Tlal}. The crucial and most novel  part of the paper is \S\ref{sec:Tlal} where we give the proof of this proposition via two decompositions of the operator into  more elementary building blocks which are combined via almost orthogonality arguments.

 \subsection*{Notation} For  nonnegative  quantities $a, b$  write $a \lesssim b$ 
 to indicate $a \leq C b$ for some constant $C$. We write  
 $a \approx  b$ to indicate $a \lesssim b$ and $b \lesssim a $. 

 \subsection*{\it Acknowledgements}  
 The second author  thanks Shaoming Guo for re-raising   the  question, Detlef M\"uller for  useful remarks on Carnot groups,     and Brian Street for conversations on general pointwise convergence issues.  We are indebted to the anonymous referee for a thorough and  careful reading of the manuscript and for very helpful suggestions to improve it.  
 The first author was supported in part by KIAS Individual Grant MG087001 during a research stay at UW Madison. 
The second author was  supported in part by NSF Grant DMS-2054220.

\section{Preliminary reductions} \label{sec:prel}

In the following theorem (which will be proved in subsequent sections) we formulate the main results, with a hypothesis of linear independence on the skew-symmetric matrices entering in the group structure. We then  show how the proof of Theorems \ref{thm:stronger} and \ref{thm:rwt} are  reduced to this result. 
\begin{thm}\label{thm:main-basic}
    Let $S_1,\dots, S_m$  be a  linearly independent set of $d\times d$ skew symmetric matrices.  Let  $\cU$ 
    be a neighborhood of the origin of $\bbR^{d-1} $ and let 
$g : \cU\to \bbR$ be a $C^\infty$ function satisfying  $g(0) = 1$, $g'(0) = 0$, $g''(0)$ positive-definite.
There exists   $\rho>0$ such that the following holds for $C^\infty_c$ functions $\beta_0$ supported in a ball  $\cU_{\rho}\subset \cU$ of diameter $\rho$ centered at the origin in $\bbR^{d-1}$: 
Let $\Gamma_g(\om')= g(\om')e_1+\sum_{i=2}^d \om_i e_{i}$, and define 
$\cA f(x,t)\equiv \cA_t f(x)$ by 
\Be\label{eq:Atdef}  \cA f(x,t) =  \int f\big(\ux-t\Gamma_g(\om'), \overline x- t\sum_{i=1}^m \overline e_{i}\ux^\intercal
S_i \Gamma_g(\om') \big) \beta_0(\om') d\om'\,.\Ee
Let $u\in C^\infty_c((\tfrac 14,4))$. Then for $p>\tfrac d{d-1}$ 
the inequality 
\Be \label{eq:main-st}
\Big\| \sup_n \|u(s) \cA f (\cdot, 2^n s)\|_{B^{1/p}_{p,1} (\bbR, ds)} \Big\|_{L^p(\bbR^{d+m})} \le C\|f\|_{L^p(\bbR^{d+m})}
\Ee
holds for all   functions $f\in L^p(\bbR^{d+m})$. Here $C$ depends on  $p$  but is independent of $f$, and independent of $\beta_0$ 
as
$\beta_0$ ranges over a bounded subset of $C^\infty_c (\cU_\rho)$.
Moreover for a compact interval $I\subset ( 0,\infty)$ 
\Be\label{eq:main-rwt} 
\big\|\mathrm{ess\,sup}_{t\in I} |\cA_t f|\big\|_{L^{p,\infty } (\bbR^{d+m})} \le C_I \|f\|_{L^{p,1} (\bbR^{d+m} )} , \quad p= \tfrac{d}{d-1}.
\Ee 
\end{thm}


We shall now show in several steps how Theorems  \ref{thm:main}, \ref{thm:stronger} and \ref{thm:rwt} are implied by Theorem \ref{thm:main-basic}.  In our first reduction we reduce to the situation that the manifold $\Sigma\subset \bbR^m$ that supports the measure $\mu$ is given as a graph with the property required in Theorem \ref{thm:main-basic}.

We fix a point $y_\circ \in \Sigma$ and consider  the operator  
$f\mapsto f* (\chi^{y_\circ} \mu)_t$ where $\chi^{y_\circ}  $ is a $C^\infty_c$ function supported in a neighborhood of $y_\circ$. It suffices to prove the analogues of Theorem \ref{thm:main} and \ref{thm:stronger} for these convolutions; once this is 
 achieved one  can use a compactness and partition of unity argument to deduce Theorems \ref{thm:main}, \ref{thm:stronger} in their original formulation. Let $e_{\circ,1}= y_\circ/|y_\circ|$ (recall that the origin lies in the domain surrounded by  $\Sigma$ and thus $y_\circ \neq 0$). Pick unit vectors $e_{\circ,i}$, $2\le i\le d$ so that $e_{\circ,1},\dots, e_{\circ,d}$ is an orthonormal basis of $\bbR^d$. As $e_{\circ,1}$ does not belong to the tangent space to $|y_\circ|^{-1} \Sigma$ at $e_{\circ,1}$  we may parametrize $|y_\circ|^{-1}\Sigma$ near $e_{\circ,1}$ by \[\Gamma(\om')= G(\om') e_{\circ, 1} +\sum_{i=2}^d \om_i e_{\circ,i};\] here $(\om')^\intercal=(\om_2,\dots, \om_d)$ and the function $G$ satisfies 
\[\text{ $G(0)=1$, $G'(0)=b$ and 
$G''(0)$  positive definite.}\]
We then have for  $\nu=\chi^{y_\circ} \mu$ 
\[ f*\nu_t(x)=\int f(\ux-t|y_\circ|\Gamma(\om'), \ox-t|y_\circ| \ux^\intercal \vec J\Gamma(\om')) \chi_\circ(\om') 
d\om' \]
with  $\chi_\circ(\om')=\chi^{y_\circ}(|y_\circ|\Gamma(\om')) |y_\circ|^d (1+|G'(\om')|^2)^{1/2} $
and $\ux^\intercal \vec J\Gamma(\om') $ denotes the vector $\sum_{i=1}^m \overline e_{i} \ux^\intercal J_i \Gamma(\om') $.

Let $P$ be the $(d-1)\times d$ matrix defined by $P=\begin{pmatrix} 0&I_{d-1}\end{pmatrix} $, corresponding to the projection $(x_1,x_2,\dots, x_d)^\intercal \mapsto (x_2,\dots,x_d)^\intercal$.
Let $R$ denote the rotation satisfying $Re_{\circ,i}=e_i$ for $i=1,\dots, d$. Then 
\[R\Gamma(\om')= \Gamma_G(\om'):=G(\om')e_1+ \sum_{i=2}^d\om_i e_i \equiv G(\om') e_1+ P^\intercal \om'.\]
Setting
$\widetilde J_i= |y_\circ|^2 RJ_iR^\intercal $ we define $\cA^{[1]}f(x,t)\equiv \cA^{[1]}_tf (x) $ by 
\Be\label{eq:A1def} \cA_t^{[1]} f(x)= \int f(\ux-t\Gamma_G(\om'), \ox-t \sum_{i=1}^m \overline e_{i}  \ux^\intercal \widetilde J_i \Gamma_G(\om') ) \chi_\circ(\om') d\om'. \Ee
Then we compute 
\[ f*\nu_t (x)= \cA_t^{[1]}h ( |y_\circ|^{-1} R\ux, \ox) , \text{ with } h(\ux,\ox)= f(|y_\circ|  R^\intercal \ux, \ox).\]
Since $R^{-1}=R^\intercal$ it suffices to prove \eqref{eq:stronger} and the maximal bounds with   $\cA^{[1]} $ in place of $A$.

We now use another transformation to reduce to the situation in Theorem \ref{thm:main-basic}.
To this end we set $g(\om')= G(\om')-b^\intercal \om'$ so that $g'(0)=0$ as in Theorem \ref{thm:main-basic}  (recall $b= G'(0) $). 

We then have (splitting $\underline x=(x_1,x')\in \bbR\times\bbR^{d-1}$)
\begin{align*}
    \cA_t^{[1]} f (x)&= \int f(x_1-t g(\om') -t b^\intercal \om', x'-t\om', \ox- t \overline v(\ux,\om') ) \chi_\circ(\om') d\om'
\end{align*}
where $\overline v= (\overline v_1,\dots, \overline v_m)$ with
\[ \overline v_i(\ux,\om')=\ux^\intercal \widetilde J_i (g(\om')+b^\intercal \om') e_1+\ux^\intercal \widetilde J_i P^\intercal \om'.\]
Now write 
\Be\label{eq:firstentry}  x_1-tg(\om') -t b^\intercal\om'=  x_1- b^\intercal x' -t g(\om')   + b^\intercal(x'-t\om')\Ee 
and \begin{align*} \overline v_i(\ux,\om')=& (x')^\intercal P\widetilde J_i e_1 g(\om')+ (x')^\intercal P\widetilde J_ie_1 (b^\intercal \om')
+(b^\intercal x') e_1^\intercal \widetilde J_i P^\intercal \om'
\\&+(x_1-b^\intercal x') e_1^\intercal \widetilde J_iP^\intercal \om' + (x')^\intercal P\widetilde J_i P^\intercal \om'.
\end{align*}

 Observe that, with  $\Gamma_g(\om')=g(\om)e_1+P^\intercal \om' $, 
\begin{align*}
   & (x')^\intercal P \widetilde J_i e_1 (b^\intercal \om')= - \ux^\intercal  P^\intercal P \widetilde J_i^\intercal e_1 b^\intercal P \Gamma_g(\om')
   \\
   &(b^\intercal x') e_1^\intercal \widetilde J_i P^\intercal \om' = 
   \ux^\intercal P^\intercal  b e_1^\intercal  \widetilde J_i P^\intercal  P  \Gamma_g(\om'),
   \end{align*}
   also the analogous formulas remain true if on the right hand sides $\ux$ is replaced with $(x_1-b^\intercal x')e_1+P^\intercal x'$. Furthermore 
   \begin{multline*} (x')^\intercal P\widetilde J_i e_1 g(\om)
+(x_1-b^\intercal x') e_1^\intercal \widetilde J_iP^\intercal \om' + (x')^\intercal P\widetilde J_i P^\intercal \om'\\= \big[(x_1-b^\intercal x')e_1+P^\intercal x' ]^\intercal  \widetilde J_i \Gamma_g(\om'). \end{multline*}
We combine the above observations, and setting
\Be\label{eq:scriptJ} \cJ_i=  \widetilde J_i+ P^\intercal be_1^\intercal \widetilde J_iP^\intercal P - P^\intercal P \widetilde J_i^\intercal e_1 b^\intercal P\Ee 
we see that $\cJ_i$ are skew symmetric $d\times d$ matrices satisfying $\cJ_i e_1=\widetilde J_i e_1$, and  that 
\Be \label{eq:vbar} \overline v_i(\ux,\om')=  
\big[(x_1-b^\intercal x')e_1+P^\intercal x' ]^\intercal  \cJ_i \Gamma_g(\om'). 
\Ee
Now if we define $\cA^{[2]}_t$ by  
\Be\label{eq:A2def} 
\cA_t^{[2]} f(x)= 
 \int f(\ux-t\Gamma_g(\om'), \ox-t \sum_{i=1}^m \overline e_{i}  \ux^\intercal \cJ_i \Gamma_g(\om') ) \chi_\circ(\om') d\om'
 \Ee then it follows from 
\eqref{eq:firstentry} and \eqref{eq:vbar} that 
\Be\label{eq:A1-and A2} \begin{aligned}
&\cA_t^{[1]}f(x_1,x',\ox)= \cA_t^{[2]} f_b (x_1-b^\intercal x', x', \ox)
\\
&\quad\text{ with  }  f_b(y_1,y', \oy)= f(y_1+b^\intercal y',y',\oy).
\end{aligned}
\Ee
Hence the desired bounds for the families $(\cA_t^{[1]})_{t>0}$ and 
$(\cA^{[2]}_t)_{t>0} $ 
are equivalent.
For the case that the matrices $\cJ_1,\dots, \cJ_m$ are linearly independent \eqref{eq:stronger} and the $L^p$ boundedness of the maximal operator in theorem \ref{thm:main} can now be obtained from Theorem \ref{thm:main-basic} (using $S_i=\cJ_i$ in that theorem). 

In the other extreme case, when all $\cJ_i$ are the zero-matrices the $L^p$ boundedness of the maximal operator operator (and the analogue of \eqref{eq:stronger}) follows  
by an application of the spherical maximal  theorems  in the Euclidean case in $\bbR^d$ (\cite{SteinPNAS1976}) and integration in the vertical variables. 
In this case we also have the result for $d=2$ by using Bourgain's theorem \cite{Bourgain-sphmax} (although this is  not needed in our proof). If $d\ge 3$ the restricted weak type inequality for $p=\frac{d}{d-1}$ can be deduced from \cite{Bourgain-CompteR} and a slicing argument.  These slicing arguments also apply to the variants where a $B^{1/p}_{p,1}$-norm is used  on the dilation parameter.

It remains to consider the case where the matrices $\cJ_i$ are not all zero but are linearly dependent. 
For this  case we need a further reduction.

\begin{lem} \label{lem:aux-slicing}
   Assume that $\cJ_1,\dots, \cJ_m$  are not all zero.  
   Then there exist   linearly independent skew symmetric $d\times d$ matrices $S_1,\dots, S_n$,  with $1\le n\le m$, and an orthogonal matrix $V\in O(m)$ such that for $\cA_t^{[2]}$ as in \eqref{eq:A2def} 
   \Be\label{eq:change-of-var}
    \cA_t^{[2]} f(x)= \cA_t^{[3]}  f_V(\underline x, V\overline x),  \Ee
    with  $f_V(y) =f(\underline y, V^\intercal \overline y)$ and 
\Be\label{eq:A3def} 
\cA_t^{[3]} f(x)= 
 \int f(\ux-t\Gamma_g(\om'), \ox-t \sum_{i=1}^n \overline e_{i}  \ux^\intercal S_i \Gamma_g(\om') ) \chi_\circ(\om') d\om'\,.
 \Ee 
   \end{lem}
\begin{proof}
Consider a basis $E_1,\dots E_{\frac{d(d-1)}2} $ in the space of $d\times d$ skew symmetric matrices.
    We can express the $\cJ_i$ in terms of the basis matrices, and obtain $\cJ_i=\sum_{j=1}^{\frac{d(d-1)}{2}} c_{ij} E_j$, $i=1,\dots, m$ for suitable scalars $c_{ij}$. We denote by   $C$  the $m\times \frac{d(d-1)}{2}$ matrix whose $(i,j)$ entry is given by $c_{ij}$. We apply the singular value decomposition of the transposed
    matrix $C^\intercal$. That is,  we decompose  
    $ C^\intercal = UDV$ where $U$ is an orthogonal $\frac{d(d-1)}{2}\times \frac{d(d-1)}{2}$ matrix, $V$ is an orthogonal $m\times m$ matrix and $D$ is a $\frac{d(d-1)}{2}\times m$ matrix  such that 
    \[D_{ij}=\begin{cases} 
    s_i&\text{ if $1\le i=j\le n$}\\ 0 &\text{ otherwise}.\end{cases} \] Here $n\le \min\{\frac{d(d-1)}{2} ,m\} $ and $s_1\ge \dots\ge s_n>0$ are the singular values. For the coefficients of $C$ we  then get 
    \[ c_{ij} = (C^\intercal)_{ji} = 
    \sum_{k=1}^{\frac{d(d-1)}{2} }\sum_{\ell=1}^m U_{jk} D_{k\ell} V_{\ell i} = \sum_{k=1}^{ n}  U_{jk} s_k V_{ki}.  \]
Defining \[S_k= s_k \sum_{j=1}^{\frac{d(d-1)}2 } U_{jk} E_j , \quad k=1,\dots n, \] it is  clear by the invertibility of $U$ that $S_1,\dots, S_n$ are linearly independent skew symmetric matrices and we obtain
\[ \cJ_i=\sum_{j=1}^{\frac{d(d-1)}{2} } c_{ij} E_j = \sum_{j=1}^{\frac{d(d-1)}{2} }\sum_{k=1}^n U_{jk}s_k V_{ki} E_j =
\sum_{k=1}^n V_{ki} S_k \,, \quad i=1,\dots,m.
\]
Hence   (using $V^\intercal =V^{-1}$) 
\begin{align*} &\overline x-t \sum_{i=1}^m \overline e_i \underline x^\intercal \cJ_i \Gamma_g(\om')
=V^\intercal \big[ V\overline x -t \sum_{k=1}^n \overline e_k \underline x^\intercal S_k \Gamma_g(\om') \big] 
\end{align*}
which gives \eqref{eq:change-of-var}.
\end{proof}

By Lemma \ref{lem:aux-slicing} the desired bounds for
$\cA_t^{[2]} $ and $\cA_t^{[3]}$ are equivalent. We show   how Theorem \ref{thm:main-basic} yields the analogue of Theorem \ref{thm:stronger} for the family $(\cA_t^{[3]})_{t>0}$ in place of $(A_t)$. 
In $\bbR^m$ we split variables  $\overline x= (\tilde x, \breve x) \in \bbR^n\times \bbR^{m-n}$. 
For $h\in L^p(\bbR^{d+n})$ we define $\fA h(x,t)=\fA_t h(x)$ by 
\[\fA_t h(\ux,\tilde x)= 
 \int h(\ux-t\Gamma_g(\om'), \tilde x-t \sum_{i=1}^n \overline e_{i}  \ux^\intercal S_i \Gamma_g(\om') ) \chi_\circ(\om') d\om'.\]
 We get from Theorem \ref{thm:main-basic} applied with $n$ in place of $m$, that  \[\Big\|\sup_{\ell\in \bbZ} \| u(s) \fA_{2^\ell s}h\|_{B^{1/p}_{p,1}(\bbR, ds)} \Big\|_{L^p(\bbR^{d+n})}  \lc \|h\|_{L^p(\bbR^{d+n})}, \quad p>\frac{d}{d-1} .\]

 Let
 $f^{\breve x} (\ux, \tilde x)= f(\ux, \tilde x, \breve x) $, and observe that
 $\cA_t^{[3]} f(\ux, \tilde x, \breve x)= \fA_t f^{\breve x} (\ux, \tilde x)$.
 We apply the $L^p(\bbR^{d+n})$-boundedness result stated above to the functions
 for 
 $h=f^{\breve  x}$ and get
 \[\iint \sup_{\ell \in \bbZ} \|u(s) \cA_{2^\ell s } ^{[3]} f (\ux,\tilde x, \breve x) \|_{B^{1/p}_{p,1} (\bbR, ds)} ^p d\ux d\tilde x
 \le C^p \iint| f(\ux, \tilde x, \breve x)|^p  d\ux d\tilde x,\]
 with $C$ independent of $\breve  x$. Integrating over  
 $\breve  x\in \bbR^{m-n}$ gives the desired result.

 We have now reduced the proof of Theorem \ref{thm:main} to the inequalities \eqref{eq:main-st} in Theorem \ref{thm:main-basic}. The above arguments also reduce (after minor modifications) the proof of Theorem \ref{thm:rwt} to the proof of inequality \eqref{eq:main-rwt}. 
 For the remainder of the paper we will be concerned with the proof of  
 Theorem \ref{thm:main-basic}.

 \begin{remarka} 
 The shear transformation showing the equivalence of the $L^p$ boundedness of the operators  associated with $\cA^{[1]}$ and $\cA^{[2]}$ is not needed for the spherical case 
 $\Sigma=S^{d-1}$, when $b=0$. 
 However in the general case it seems necessary, and  we 
 take this opportunity to correct an inaccuracy in \cite{MuellerSeeger2004}  which deals with the case of M\'etivier groups (i.e. the matrices $\sum_{i=1}^m c_i J_i$ are invertible if  $(c_1,\dots, c_m)\neq 0$). There it is stated that this reduction follows for more general $\Sigma$ by a rotation argument which is not the case. One can use the above shear transformations instead and deduce that the arguments in \cite{MuellerSeeger2004} apply to surfaces $\Sigma$ that are small perturbations of the sphere. Such a   perturbation assumption would be needed  for the proof in \cite{MuellerSeeger2004} since 
 the M\'etivier condition on the matrices $J_i$ (and, equivalently,  on the   $\widetilde J_i=|y_\circ|^2RJ_iR^\intercal$)   guarantees  the M\'etivier condition on the matrices $\cJ_i$ in \eqref{eq:scriptJ} only when $b$ is sufficiently small.  In the setup of this paper there is no such small smallness assumption on $b$ needed.
 \end{remarka}

 For the remainder of the paper we will give the proof of Theorem \ref{thm:main-basic} and fix linearly  independent  skew symmetric matrices $S_1,\dots, S_m$.
  For later use 
  notice that this assumption implies that there is a  $c_0>0$  such that 
    \Be\label{eq:lincombofSi} c_0\le \Big\| \sum_{i=1}^m {\theta_i}S_i \Big\| \le c_0^{-1}  \text{ for all }\theta\in \bbR^{m} \text{ with } 1/4\le|\theta|\le 4.\Ee
 This is immediate from the  fact that $\theta\to \|\sum_{i=1}^m\theta_i S_i\|  $ is a 
 continuous function which takes a minimum and a maximum on the annulus $\{\theta: 1/4\le|\theta|\le 4\}$, and by the assumed linear independence this minimum is positive.

\section{Dyadic frequency decompositions}\label{sec:dyadic}
We now use the group structure on $\bbR^{d+m}$ given by \eqref{eq:grouplaw} but with the  $J_i$ replaced by the skew symmetric matrices $S_i$, {with  $S_1,\dots, S_m$ linearly independent.}
We denote by $\nu$ the measure defined by $\inn{\nu}{f}=\int f(g(\om') ,\om',0)d\om'$
which can also be written as the pairing of the distribution  \[\beta_0(x') \beta_1(x_1,\ox)  \delta(x_1-g(x'), \ox)\] with $f$; here $\delta$ is the Dirac measure  in $\bbR^{m+1}$, $\beta_0$ is a $C^\infty$ function supported on a ball  of radius $\varrho$ centered at the origin of $\bbR^{d-1}$ and $\beta_1$ is a $C^\infty$ function supported on an $\vr^2$-ball centered at $(1,0,\dots,0)\in \bbR^{m+1}$, here $\epo\ll \varrho$. We assume that $\vr$ is small compared with the  reciprocal of the $C^{3}$ norm of $g$, also $\vr \ll \|(g''(0))^{-1} \|^{-1} $ and finally $\vr \ll c_0$ where $c_0$ is as in \eqref{eq:lincombofSi}.

We use a dyadic frequency decomposition of the Fourier integral of $\delta$ to decompose $\nu=\sum_{k=0}^\infty \nu^k $ where 
\Be\label{eq:oscillatorynuk} \nu^k (x)=  \frac{\beta_1(x_1,\ox) \beta_0(x') }{(2\pi)^{m+1}}\iint \zeta_k(\sqrt{\sigma^2+|\tau|^2}) e^{i \sigma (x_1-g(x'))+i \inn{\tau} {\overline x}} d\sigma d\tau 
\Ee
where $\zeta_0\in C^\infty_c(\bbR)$ is supported on  $(-1,1)$, $\zeta_0(s)=1$ for  $|s|<3/4$ and $\zeta_k(s)= \zeta_0( 2^{-k}s) -\zeta_0(2^{1-k} s)$ when $k\ge 1$; hence,  for $k\ge 1$ the function $\zeta_k= \zeta_1(2^{-(k-1)} \cdot)$ is supported in 
$(-2^{k-1}, -2^{k-3})\cup (2^{k-3}, 2^{k-1})$. For $k>0$ we make a further decomposition in the $\sigma$-variables setting 
\Be\label{eq:nukl}\nu^{k,l}  (x)=  \frac{\beta_1(x_1,\ox) \beta_0(x') }{(2\pi)^{m+1}}\iint \zeta_{k,l}(\sigma,\tau)e^{i \sigma(x_1-g(x'))+i\inn{ \tau} {\overline x}} d\sigma d\tau 
\Ee where \[ \zeta_{k,l} (\sigma, \tau)=\begin{cases} \zeta_1(2^{1-k} \sqrt{\sigma^2+|\tau|^2})\zeta_1(2^{l+1-k} \sigma) &\text{ for } l<k
\\ \zeta_1(2^{1-k} \sqrt{\sigma^2+|\tau|^2})
\zeta_0 (\sigma) &\text{ for } l=k
\end{cases} 
\] i.e., 
for $k\ge 1$, $l<k$ we have the restriction $|\sigma|+|\tau|\approx 2^k$  and $|\sigma|\approx 2^{k-l}$ in the frequency variables. 
We set $\nu^{k, l}_t(x)= t^{-d-2m} \nu^{k,l} (t^{-1}\ux, t^{-2}\ox)$ and similarly define $\nu^k_t$.

We 
state the main local estimates for $f*\nu^{k,l}_t$.  
\begin{prop} \label{prop:mainAkl}  Let   $\eps>0$. Let $I$ be a compact subinterval of $(0,\infty)$. 
  Then there exists a constant $C = C(\eps,I)>0$ such that the following holds.
 \begin{multline}\label{est:akldakl}
\Big(\int_I\|f*\nu^{k,l}_t\|_{L^p(\bbR^{d+m})}^pdt\Big)^{1/p} +2^{l-k} \Big(\int_I\|\partial_t (f*\nu^{k,l}_t)\|_{L^p(\bbR^{d+m})} ^pdt\Big)^{1/p}
\\ \le \begin{cases} C 2^{-\frac{k(d-1)}{p'}} 2^{l(\frac{d-2}{p'}+\eps)} \|f\|_{L^p(\bbR^{d+m})} \text{ if $1\le p\le 2$,}\\
C 2^{-\frac{k(d-1)}{p}} 2^{l(\frac{d-2}{p}+\eps)} \|f\|_{L^p(\bbR^{d+m})} \text{ if $2\le p<\infty$.}
\end{cases} 
\end{multline}
\end{prop}

\begin{cor}\label{cor:Bes}
 Let  $\frac{d}{d-1}<p<\infty$, $u\in C^\infty_c((0,\infty))$. Let $f\in L^p(\bbR^{d+m})$. Then for almost every $x\in \bbR^{d+m}$ the function $t\mapsto \cA f(x,t) $ is continuous, and for any $\la>0$
\Be\label{eq:Besov}
\Big(\int_{\bbR^{d+m}} \big\| u(\cdot) Af(x,\la \cdot) \big\|_{B^{1/p}_{p,1}}^p dx\Big)^{1/p} \lc \|f\|_{L^p(\bbR^{d+m})}.
\Ee
\end{cor}
\begin{proof}  By scaling we can assume that $\la=1$.  The first  statement follows from the second, since $B^{1/p}_{p,1}$ embeds into the space of bounded continuous functions. 
 Let $1\le p\le 2$. Set 
 $\cR^{k,l} f(x,s)=u(s) f*\nu^{k,l}_s(x) $.  We use the interpolation inequality $\|g\|_{B^\theta_{p,1}} \lc \|g\|_p^{1-\theta}\|g'\|_p^\theta$ ($0<\theta<1$),  H\"older's inequality, Fubini and the proposition  
to deduce that the left hand side of $\eqref{eq:Besov}$ is dominated by 
\begin{align*}
 \|\cR^{k,l} f\|_{L^p(B^{1/p}_{p,1})} &\lc  \| \cR^{k,l} f\|_{L^p(L^p)}^{1-\theta}
     \|\partial_t \cR^{k,l}  f\|_{L^p( L^p)}^{\theta}
\\&\lc 2^{-k(\frac{d-1}{p'}-\theta)} 2^{l(\frac{d-2}{p'}-\theta+\eps)} \|f\|_p.\end{align*}
The desired inequality
follows by summing over $l\le k$ and then summing over $k$ (which is possible if $d\ge 3$ and $\frac{d}{d-1}<p\le 2$, $\theta=1/p$).  A similar argument applies for $p>2$.
\end{proof} 
\begin{proof}[Proof of the restricted weak type inequality in Theorem \ref{thm:main-basic}] 
Let  $\sR^{k,l}f(x,t):=f*\nu^{k,l}_t$ and as in the proof of Corollary \ref{cor:Bes} we have that $\sR^{k,l}$ maps $L^p$ to $L^p(L^\infty)$ with operator norm 
$O( 2^{k(\frac 1p-\frac{d-1}{p'}) } 2^{-l (\frac 1p- \frac{d-2}{p'} -\eps)})$.  If  $1\le p< \frac{d-1}{d-2}$ we may (for sufficiently small $\eps$) sum in $l$ and obtain in this range 
\[ 
\|\mathrm{ess\,sup}_{t\in I} |f*\nu^k_t \|_{p} \lc 2^{k(\frac 1p-\frac{d-1}{p'}) } \|f\|_p.
\]
We are now applying the `Bourgain trick' in \cite{Bourgain-CompteR} to sum in $k$ and deduce that \[ 
\|\mathrm{ess \, sup}_{t\in I}  |f*\nu_t|  \|_{L^{p,\infty} } \lc  \|f\|_{L^{p,1}}, \quad p=\tfrac{d}{d-1}. \qedhere
\] 
\end{proof}

The most interesting part of Proposition \ref{prop:mainAkl} is the $L^2$-estimate. The $L^p$ estimates  follow by interpolation with  $L^1$ estimates which we now briefly  discuss.

\subsection*{$L^1$ and $L^\infty$  estimates} In what follows $\beta(x)=\beta_1(x_1,\ox)\beta_0(x')$.
By integration by parts with respect to $\sigma,\tau$ we  obtain the inequality
\Be\label{eq:pointwisebd}
|\nu^{k,l}(x)| \lc_N \frac{2^{k-l}}{(1+2^{k-l}|x_1-g(x')|)^N} \frac{2^{km}}
{(1+2^k|\ox|)^N};
\Ee
moreover  $2^{-k} \nabla \nu^{k,l}$, $2^{l-k} \partial_s \nu^{k,l}_s $, 
$2^{l-2k} \partial_s \nabla \nu^{k,l}_s $ satisfy for $|s|\approx 1$ the same pointwise bounds.
 Hence we obtain 
 \Be \label{eq:L1bound} \|\nu^{k,l}\|_1 + 2^{l-k}\|\partial_s \nu^{k,l}_s\|_1 \lc 1.\Ee
For later use we also record 
\Be \label{eq:derivatives-nukl}\|\nabla \nu^{k,l}\|_1 + 2^{l-k}\|\nabla \partial_s \nu^{k,l}_s\|_1 \lc 2^k.\Ee


We will show in the next section  \S\ref{sec:global} how to prove the $L^p$ boundedness for the global maximal operator and its strengthening in Theorem \ref{thm:stronger}, given the result of Proposition \ref{prop:mainAkl} and \eqref{eq:L1bound}, \eqref{eq:derivatives-nukl}. The proof of Proposition \ref{prop:mainAkl} will be given in \S\ref{sec:propprel} -\S\ref{sec:Tlal}.

\section{The global maximal operator}\label{sec:global}
We  prove the global bound in Theorem \ref{thm:stronger},  given Proposition \ref{prop:mainAkl}. 
The reduction to  Proposition 
\ref{prop:mainAkl} follows closely  arguments in \cite{MuellerSeeger2004}; we include  details  for the convenience of the reader.

Let $I=(\frac 14,4)$ and $u\in C^\infty_c(I)$. We will prove the estimate 
\begin{multline}\label{eq:tobeproven1} \Big\| \sup_{n\in \bbZ} \big\| u(s) f*\nu^{k,l}_{2^n s} \|_{B^{1/p}_{p,1}} \Big\|_p \\\le C_\eps \begin{cases} 
(1+k)^{1/p} 2^{ k (\frac 1p-\frac{d-1}{p'})} 2^{l(\frac{d-2}{p'}-\frac{1}{p}+\eps)} \|f\|_p, &1<p\le 2
\\
(1+k)^{1/p} 2^{ k d/p}  2^{l(\frac{d-1}{p} +\eps )} \|f\|_p, &2\le p<\infty, 
\end{cases} \end{multline}
for $0\le l\le k$; here the Besov-norm  
is taken with respect to the $s$ variable.  Summing in $l,k$ for $p>\frac{d}{d-1}$ implies \eqref{eq:stronger}. 

Recall that $\nu^{k,l}$ is compactly supported and that \[ \Big|\int \nu^{k,l} (x) dx\Big| \lc_N 2^{-kN};\]
this is seen by  using \eqref{eq:nukl} and repeated integration by parts,  with respect to  $(x_1,\ox) $,  if $l$ is small and with respect to $\ox$ if $l$ is large. 

As noted in \cite{MuellerSeeger2004} we can write \[\nu^{k,l}(\ux,\ox)= \cK^{k,l}(\ux,\ox)+ \gamma_{k,l} \rho(\ux,\ox)\] where $\rho \in C^\infty_c(\bbR^{d+m})$ function supported near the origin, $|\gamma_{k,l}|\le c_N 2^{-kN}$ for $0\le l\le k$ and \Be\label{canc} \int \cK^{k,l}(x) d x=0.\Ee 
Set  $\cK^{k,l}_{t}(\ux,\ox) = t^{-d-2m}\cK^{k,l}(t^{-1}\ux,t^{-2}\ox)$, 
$\rho _t(\ux,\ox)= t^{-d-2m} \rho(t^{-1}\underline{x}, t^{-2} \overline {x} )$. 

The contribution from $\ga_{k,l} \rho$ is harmless,  indeed for all $n\in \bbZ$, $s\in [1/4,4] $ with  $j=0,1,2,\dots$
\[ \Big|(\frac{d}{ds})^j\big[ u(s)   f* \rho_{2^n s} (x)\big]\Big|\lc c_j  M f(x)\]
where $M$ is the Hardy-Littlewood maximal operator associated to the Carnot-balls on the group $G$. This implies 
\[ \Big\| \sup_{n\in \bbZ} \| u(s) f*\rho_{2^ns} \gamma_{k,l}  \|_{B^{1/p}_{p,1}} \Big\|_p \le C_N 2^{-kN} \|f\|_p. \]

Therefore we need to show the equivalent of \eqref{eq:tobeproven1} where $\nu^{k,l}$ is replaced by $\cK^{k,l}$. This is implied by two stronger inequalities where the $\sup$ in $n$ is replaced by an $\ell^2$ norm in $n$ when $1<p\le 2$ and by an $\ell^p$ norm in $n$ when $2<p<\infty$.

We shall prove   for $1<p\le 2$
\Be\label{eq:tobeproven2}
\Big\| \Big(\sum_{n\in \bbZ} \| u (s)f*\cK^{k,l}_{2^n s}\|_{B^{1/p}_{p,1}}^2\Big)^{1/2} \Big\|_p \le C_{\eps,p}
(1+k)^{1/p} 2^{ k (\frac 1p-\frac{d-1}{p'}) } 2^{l(\frac{d-2+\eps}{p'}-\frac{1}{p})} \|f\|_p
\Ee and for $2\le p<\infty$
\Be\label{eq:tobeproven3}
\Big\| \Big(\sum_{n\in \bbZ} \| u(s) f*\cK^{k,l}_{2^n s} \|_{B^{1/p}_{p,1}}^p\Big)^{1/p} \Big\|_p \le C_{\eps,p}
(1+k)^{1/p} 2^{ -k\frac{d-2}p} 2^{l\frac{d-3+\eps}{p}} \|f\|_p
\Ee 
We use the  interpolation inequality 
\Be \label{Besov-embedding}  \|g\|_{B^{1/p}_{p,1}} \lc \|g\|_p + \|g\|_p^{1- 1/p} \|g'\|_p^{ 1/p}
\Ee 
which is elementary and also expresses the  identification of $B^{1/p}_{p,1}$ as the real interpolation space $[L^p, W^{1,p}]_{1/p,1}$.

We focus on the case $1\le p\le 2$ and prove \eqref{eq:tobeproven2}.  The  embedding inequality implies 
\begin{multline*} 
\Big\| \Big(\sum_{n\in \bbZ} \| u(s) f*\cK^{k,l}_{2^n s} \|_{B^{1/p}_{p,1}}^2\Big)^{1/2} \Big\|_p \lc \Big\| \Big( \sum_{n\in \bbZ} \Big( \int_I|f*\cK^{k,l}_{2^ns}|^p ds\Big)^{2/p}\Big)^{1/2}\Big\|_p \\ +
\Big\| \Big(\sum_n
\Big(\int_I|f* \cK_{2^ns}^{k,l}|^ pds\Big)^{\frac{2}{pp'} }
\Big(\int_I\big| f* \tfrac{d}{ds}\cK_{2^ns}^{k,l}\big|^ pds\Big)^{\frac{2}{p^2} }
\Big)^{1/2} \Big\|_p  =:\cE_1+\cE_2\,.
\end{multline*}
For the first expression on the right hand side we have by the integral Minkowski's inequality for $\ell^{2/p}$
\[\cE_1\le  \Big\|\Big(\int_I \Big(\sum_n |   f* \cK_{2^ns}^{k,l}  |^2\Big)^{p/2} ds \Big)^{1/p} \Big\|_p  \]  and for the second term  we use  applications of H\"older and then the integral Minkowski inequality
\begin{align*}
    &\cE_2\le \Big\| 
    \Big( \sum_n\Big( \int_I| f*\cK^{k,l}_{2^ns} |^pds\Big)^{\frac 2p}\Big)^{\frac{1}{2p'}}
\Big( \sum_n\Big( \int_I| f*\tfrac{d}{ds} \cK^{k,l}_{2^ns} |^pds\Big)^{\frac 2p}\Big)^{\frac{1}{2p}}
\Big\|_p
\\
&\le \Big\| \Big( \sum_n\Big( \int_I| f*\cK^{k,l}_{2^ns} |^pds\Big)^{\frac 2p}\Big)^{\frac{1}{2}}\Big\|_p^{\frac{1}{p'} }
 \Big\| \Big( \sum_n\Big( \int_I| f*\tfrac{d}{ds} \cK^{k,l}_{2^ns} |^pds\Big)^{\frac 2p}\Big)^{\frac{1}{2}}\Big\|_p^{\frac{1}{p} }
 \\&\le
\Big\|\Big( \int_I\Big(\sum_n| f*\cK^{k,l}_{2^ns} |^2\Big)^{\frac p2}ds\Big)^{\frac 1p} \Big\|_p^{\frac{1}{p'} }
\Big\|\Big( \int_I\Big(\sum_n| f*\tfrac {d}{ds} \cK^{k,l}_{2^ns} |^2\Big)^{\frac p2}ds\Big)^{\frac 1p} \Big\|_p^{\frac{1}{p} }.
 \end{align*}
Since we may interchange the $x$ and the $s$ integration, everything for $p\le 2$ follows now from 
\begin{subequations} \label{eq:tobeproven4}
\begin{multline}\label{eq:tobeproven4a}
\Big(\iint_{G\times I}  \Big(\sum_n|f* \cK^{k,l} _{2^n s}(x)|^2\Big)^{\frac p2} dx\, ds \Big)^{\frac 1p} \\ \lc_\eps( 1+k)^{1/p}2^{-k\frac{d-1}{p'} } 2^{l(\frac{d-2}{p'}+\eps) } \|f\|_p
\end{multline}
and \begin{multline} \Big(\iint_{G\times I}  \Big(\sum_n|f* \tfrac d{ds} \cK^{k,l} _{2^n s}(x) |^2\Big)^{\frac p2} dx\, ds \Big)^{\frac 1p}\\
 \lc_\eps 2^{k-l} (1+k)^{1/p}2^{-k\frac{d-1}{p'} } 2^{l(\frac{d-2}{p'}+\eps) } \|f\|_p\,.
\end{multline}
\end{subequations}
We prove \eqref{eq:tobeproven4} by Marcinkiewicz interpolation, using the $L^2$ bounds  
\begin{subequations}\label{eq:tobeprovenL2}
\begin{multline}\label{eq:tobeproven51}
 \Big(\iint_{G\times I} \sum_{n\in \bbZ}|f*\cK^{k,l}_{2^n s}|^2 dx\,ds\Big)^{\frac 12}
\lc_\eps (1+k)^{\frac 12} 2^{-k\frac{d-1}{2} + l(\frac{d-2}{2}+\eps)} \|f\|_{L^2},
\end{multline}
\begin{multline} \label{eq:tobeproven52}
\Big(\iint_{G\times I} \sum_{n\in \bbZ}\big |f*\tfrac{d}{ds}\cK^{k,l}_{2^n s}(x)\big|^2 dx\,ds\Big)^{\frac 12}
\\ \lc_\eps 2^{k-l}  (1+k)^{\frac 12}  2^{-k\frac{d-1}{2} + l(\frac{d-2}{2}+\eps)} \|f\|_{L^2}, 
\end{multline} 
\end{subequations}
and the weak type $(1,1)$ inequalities
\begin{subequations}
\begin{multline} \label{eq:weaktype1}
\meas\big(\{(x,s)\in G\times I:\big(\sum_{n\in \bbZ} |f*\cK^{k,l}_{2^n s} (x) |^2\big)^{\frac 12}>\alpha\} \big)\\ \lc (1+k) \alpha^{-1}\|f\|_1,
\end{multline} 
\begin{multline} 
\label{eq:weaktype2}
\meas\big (\{ (x,s)\in G\times I: \big(\sum_{n\in \bbZ} |f*\frac d{ds}\cK^{k,l}_{2^n s} (x) |^2\big)^{1/2}>\alpha\} \big )\\ \lc (1+k)2^{k-l} \alpha^{-1}\|f\|_1\,.
\end{multline}
\end{subequations} The kernels $\cK^{k,l}_{2^ns}$ and $2^{l-k} \frac{d}{ds} \cK^{k,l}_{2^n s}$ 
enjoy similar qualitative and quantitative properties and therefore we shall only give the proofs of  \eqref{eq:tobeproven51} and \eqref{eq:weaktype1}; the proofs of  \eqref{eq:tobeproven52}, \eqref{eq:weaktype2} require only minor notational modifications.
\begin{proof}[Proof of \eqref{eq:tobeproven51}]
As a consequence of Proposition \ref{prop:mainAkl} and scaling  we have, for each fixed $n$,
\[\Big(\int_I \|f*\cK^{k,l}_{2^n s}\|_{L^2} ^2ds\Big)^{1/2}
\lc_\eps 2^{-k\frac{d-1}{2} + l(\frac{d-2}{2}+\eps)} \|f\|_{L^2}.
\]
Note that Proposition \ref{prop:mainAkl} was stated for $\nu_t^{k,l}$,
but clearly by the above discussion we can replace  $\nu^{k,l}_t$ with $\cK^{k,l}_t$.

To combine the estimates for different $n$ we
need the   following variant of the Cotlar-Stein lemma.
    Let $H_1, H_2$ be Hilbert spaces and let $T_n: H_1\to H_2$, $n\in \bbZ$ be bounded operators. Assume $B\ge 2A$,   and \[\|T_n\|_{H_1\to H_2}\le A, \quad  \|T_n T_{n'} ^*\|_{H_2\to H_2}\le B^2 2^{-\eps|n-n'|}
    \] for all $n,n'\in \bbZ$. Then for all $f\in H_1$
    \Be\label{eq:CotlarSteinAB}\Big(\sum_{n\in \bbZ} \|T_n f\|_{H_2}^2\Big)^{1/2} \lc_\eps A\sqrt{\log(\eps^{-1}B/A)} \|f\|_{H_1}. \Ee
     This is proved for the case $H_1=H_2$  in \cite[Lemma 3.2]{MuellerSeeger2004} but the  proof also extends  to the situation of two different Hilbert spaces. 
     
     We apply this  with $H_1=L^2(\bbR^{d+m})$ and $H_2= L^2(\bbR^{d+m}\times[1,2])$, for the operators 
     $T_n:H_1\to H_2$ given by $T_nf(x,s)= f*\cK^{k,l}_{2^n s}$.
     Clearly we have $\|T_n T_{n'}^*\|\lc_\eps A_{k,l}^2$ with  $A_{k,l} =     2^{-k(d-1)/2+l(d-2+\eps)/2}$; we use this for $|n-n'|\le 2(m+2)k.$  
     For large $|n-n'|$ we need to establish exponential decay in $|n-n'|$ but in view of the logarithmic dependence on $B$ in \eqref{eq:CotlarSteinAB} we do not have to care about any blowup in  terms of  powers of $2^k$ in such an estimate.
     
     As 
    $\nu^{k,l}_s$ and  $2^{-k} \nabla \nu^{k,l}_s$
    are for $s\approx 1$ pointwise dominated by the right hand side of \eqref{eq:pointwisebd} 
 the kernels
      $\cK^{k,l}_s$, $2^{-k} \nabla \cK^{k,l}_s$, 
      satisfy up to a constant the same bounds.
Since they are also  supported on a fixed common compact set we have 
   \Be\label{eq:variousL1estimates}\|\cK^{k,l}_s\|_1+2^{-k} \| \nabla \cK^{k,l}_s\|_1 
   =O(1)\Ee
   for $|s|\approx 1$. For the orthogonality arguments it is convenient  to just use 
the following trivial  pointwise bounds with an exponential dependence on $k$, with 
$N>\max\{d,m\}$:
\Be\label{eq:trivpointw}
\begin{aligned} |K^{k,l}(x)|+ 2^{-k} |\nabla K ^{k,l}(x)| &\lc_N 2^{k (m+1)} (1+|x|)^{-2N}
\\
&\lc_N 2^{k(m+1)} (1+|\underline x|)^{-N} (1+|\overline x|)^{-N} .
\end{aligned}
\Ee
Clearly, \eqref{eq:trivpointw} is implied by  the stronger bounds in \eqref{eq:pointwisebd}.

   A standard argument   using \eqref{eq:trivpointw} and the cancellation property \eqref{canc} gives the (non-optimal)  estimate 
   \[   \| \widetilde \cK^{k,l}_{2^n s} *\cK^{k,l}_{2^{n'} s} \|_1\lc 2^{k(2m+3)}  2^{-|n-n'|/2} ;  \] 
   here  we use the notation $\widetilde \cK(x)=\overline{\cK(-x)}$. We  refer to  \cite[ch.XIII, \S5.3]{Stein-harmonic} for a very similar calculation. Consequently 
\[ \|f* \widetilde \cK^{k,l}_{2^n s} *\cK^{k,l}_{2^{n'} s} \|_2
+ 2^{l-2k} \|f*\partial_s\widetilde \cK^{k,l}_{2^n s} *\partial_s\cK^{k,l}_{2^{n'} s} \|_2 
 \lc 2^{k(2m+3)} 2^{-|n-n'|/2} \|f\|_2.\] 
 Hence we  get 
$\|T_{n'}T_{n}^*\|
\lc B^2 2^{-|n-n'|/2}$, with $B^2=2^{k(2m+3)} $. We may thus apply the almost orthogonality inequality \eqref{eq:CotlarSteinAB} with   $\log(B/A_{k,l}) \lc 1+k$, and 
\eqref{eq:tobeproven51} follows.
\end{proof}
\begin{proof}[Proof of \eqref{eq:weaktype1}] We use a Calder\'on-Zygmund decomposition of $f\in L^1(G)$, at height $\alpha$, as described in  \cite[Ch.3A]{FollandStein}. 
The Carnot-Carath\'eodory balls $B(0,r)$ at the orgin are given by $\{y: |\uy|\le r,|\oy|\le r^2\}$, and the ball centered at some  $z$ is just the left translate, \textit{i.e.,} $B(z,r)=\{y: z^{-1}\cdot y\in B(0,r)\}$. 

One can decompose an $L^1(G)$ function $f$ as $f=g+b$ where $\|g\|_1\lc \|f\|_1$, $\|g\|_\infty\lc \alpha$, and $b=\sum_\nu b_\nu$ where $b_\nu$ are 
supported on the balls 
$B(y_\nu,r_\nu)$ which are  explicitly given by 
\[B(y_\nu,r_\nu) = 
\{(\uy,\oy): |-\uy_\nu+\uy|\le r_\nu, |-\oy_\nu+\oy-\uy_\nu^\intercal \vec S \uy|\le r_\nu^2\}.\]
Moreover, the $b_\nu$  satisfy $\int b_\nu (y) dy=0$ and  $\sum_\nu\|b_\nu\|_1\lc \|f\|_1$. Finally the $B(y_\nu, r_\nu)$ have bounded overlap and if for $A\ge 2$ we define 
 $\Om_\alpha:= \cup_\nu  B(y_\nu,A r_\nu)$,  then
\Be\label{eq:measureOm} \meas(\Omega_\alpha ) \lc A^{d+2m} \alpha^{-1} \|f\|_1.\Ee

We set  $\|\vec S\|:=\sum_{i=1}^m\|S_i\|$ (with the matrix norm associated to the Euclidean norm in $\bbR^d$) and  we will choose  $A\ge 10 (\|\vec S\|+1).$

We now turn to the estimation of \eqref{eq:weaktype1}. By Chebyshev's inequality and then \eqref{eq:tobeproven51}
\begin{align*}
&\meas\big(\{(x,s)\in G\times I:\big(\sum_{n\in \bbZ} |g*\cK^{k,l}_{2^n s} (x) |^2\big)^{\frac 12}>\alpha/2\} \big) 
\\ &\lc\alpha^{-2} \iint_{G\times I} \sum_{n\in \bbZ} |g*\cK^{k,l}_{2^ns} |^2 dx\, ds 
\lc\alpha^{-2} \|g\|_2^2 \lc \alpha^{-1}  \|f\|_1.
\end{align*} 
Moreover $\meas(\Omega_\alpha\times I) \le 4|\Omega_\alpha|\lc  \alpha^{-1}\|f\|_1.$
It remains to estimate the contribution involving  $b$, namely 
\begin{align*}
    &\meas\big(\{(x,s)\in \Omega_\alpha^\complement\times I:\big(\sum_{n\in \bbZ} |b*\cK^{k,l}_{2^n s} (x) |^2\big)^{\frac 12}>\alpha/2\} \big) 
    \\
    &\lc \alpha^{-1}\int_I\int_{\Omega_\alpha^\complement} \big(\sum_{n\in \bbZ} |b*\cK^{k,l}_{2^n s} (x) |^2\big)^{\frac 12} \, dx\,ds
    \\
    &\lc \alpha^{-1} \sum_\nu  \int_I\int_{ (B(y_\nu, A r_\nu))^\complement}\big(\sum_{n\in \bbZ} |b_\nu*\cK^{k,l}_{2^n s} (x) |^2\big)^{\frac 12} \, dx\, ds
    \\ &\lc \alpha^{-1} \sum_\nu \sup_{s\in I} \sum_{n\in \bbZ}
    \int_{(B(y_\nu, A r))^\complement} |b_\nu*\cK^{k,l}_{2^n s} (x) |\,  dx\
\end{align*}

We claim that for $s\in I$ and fixed $\nu$,
\Be\label{eq:L1claim}
\int_{(B(y_\nu, A r))^\complement} |b_\nu*\cK^{k,l}_{2^n s} (x) |\,  dx
\lc \begin{cases}
    1& \text{ for all $n\in \bbZ$}
    \\
    2^{k(m+1)} 2^{n}r_\nu^{-1}  &\text{ for $n\le \log_2 r_\nu$}
    \\ 2^{k(m+2)} (2^{-n} r_\nu)^{1/2} &\text{ for $n\ge \log_2 r_\nu$}
\end{cases}
\Ee 
We use the first bound when $\log_2 r_\nu -10 km\le n\le \log_2 r_\nu+ 10 k m$,  the second bound when $n< \log_2r_\nu -10 km$ and the third when $n> \log_2 r_\nu +10 km$. Summing these yields 
\[\sup_{s\in I} \sum_{n\in \bbZ} \int_{(B(y_\nu, A r))^\complement} |b_\nu*\cK^{k,l}_{2^n s} (x) |\,  dx\lc (1+k)\|b_\nu\|_1.\]
If we then sum over $\nu$ and use $\sum_\nu\|b_\nu\|_1\lc \|f\|_1$  above we  get 
\[\meas\big(\{(x,s)\in G\times I:\big(\sum_{n\in \bbZ} |g*\cK^{k,l}_{2^n s} (x) |^2\big)^{\frac 12}>\alpha/2\} \big) \lc (1+k)\alpha^{-1}\|f\|_1.\]

We now prove \eqref{eq:L1claim}.
The $O(1)$ bound in \eqref{eq:L1claim} is immediate and follows from $\|\cK^{k,l}_{2^n s}\|_1=O(1)$ which is a consequence of \eqref{eq:L1bound}.  For the second and third case in \eqref{eq:L1claim} we
will just use the trivial pointwise bounds in \eqref{eq:trivpointw}.

We now assume that $n\le \log_2 r_\nu$ to prove the second bound in \eqref{eq:L1claim}. Here we will strongly use that the integration is extended over the complement of $(B(y_\nu, Ar_\nu))^\complement$ and split it as $X_{1}\cup X_{2} $ where
\begin{align*} X_{1}&=\{(\ux,\ox): |\ux-\underline y_\nu| \ge  Ar_\nu \},
\\
X_{2} &= \{ (\ux,\ox): |\ux-\underline y_\nu|\le Ar_\nu, |-\oy_\nu+\ox-\uy_\nu^\intercal\vec S \ux|\ge A^2r_\nu^2\}.
\end{align*}
Then
\begin{multline*} 
\int_{X_{1}} |b_\nu*\cK^{k,l}_{2^n s} (x)|dx\,\lc \, 
2^{k(m+1)} \int_{B(y_\nu,r_\nu) }|b_\nu(y)| \,\,\times\\ \int_{|\underline x-\underline y_\nu|\ge Ar_\nu } \frac{(2^ns)^{-d} } { (\frac{|\underline x-\underline y|}{2^ns} )^N } 
\int_{\overline x}  \frac{(2^n s)^{-2m} }
{ (1+\frac{ | \overline x-\overline y +\underline x^\intercal\vec S\underline y| } {(2^n s)^2} )^{N} } d\overline x\,d\underline x\, dy\,.
\end{multline*}
The $\overline x$-inner integral is $O(1)$ and in the $\underline x$-integral we can use $|\underline x-\underline y|\approx|\underline x-\underline y_\nu| $ and  get a bound
$(Ar_\nu (2^ns)^{-1} )^{d-N}$ here.  Hence 
\Be\int_{X_{1}} |b_\nu*\cK^{k,l}_{2^n s} (x)|dx\lc 2^{k(m+1)} (2^{-n} r_\nu)^{d-N} \lc 2^{k(1+m)} 2^nr_\nu^{-1} \|b_\nu\|_1.
\Ee  For the $X_{2}$-contribution 
\begin{multline} \label{eq:X2-integral}
\int_{X_{2}} |b_\nu*\cK^{k,l}_{2^n s} (x)|dx\,\lc \, 
2^{k(m+1)} \int_{B(y_\nu,r_\nu) }|b_\nu(y)| \,\,\times\\ \int_{|\underline x-\underline y_\nu|\le Ar_\nu } \frac{(2^ns)^{-d} } { ( 1+\frac{|\underline x-\underline y|}{2^ns} )^N } 
\int_{|\overline x-\overline y_\nu+  \underline x^\intercal \vec S \underline y_\nu| \ge A^2 r_\nu^2}  \frac{(2^n s)^{-2m} }
{ (\frac{ | \overline x-\overline y +\underline x^\intercal\vec S\underline y| } {(2^n s)^2} )^{N} } d\overline x\,d\underline x\, dy\,.
\end{multline}
Here we gain in the $\overline x$-integral. For this observe for $x\in X_2$, $y\in B(y_\nu,r_\nu)$
\begin{align*} &\Big| |\overline x-\overline y +\underline x^\intercal\vec S\underline y| 
-
|\overline x-\overline y_\nu +\underline x^\intercal\vec S\underline y_\nu|\Big| 
\le |\overline y_\nu-\overline y+ \underline x^\intercal \vec S(\underline y-\underline y_\nu)|
\\ &\le|\overline y_\nu-\overline y+\underline y_\nu^\intercal \vec S(\underline y-\underline y_\nu)|+ |(\underline x-\underline y_\nu) ^\intercal\vec S (\underline y-\underline y_\nu)|
\\&\le |\overline y-\overline y_\nu+\underline y^\intercal\vec S \underline y_\nu|+|\underline x-\underline y_\nu| \|\vec S\| |\underline y-\underline y_\nu| \le r_\nu^2+ A\|\vec S\| r_\nu^2\le (Ar_\nu)^2/2
\end{align*}
(here we used that $|\underline x-\underline y_\nu|\le Ar_\nu$ in $X_{2}$ and $A\gg\|\vec S\|$). The displayed inequality tells us that 
we can replace 
$|\overline x-\overline y +\underline x^\intercal\vec S\underline y|$ with 
$|\overline x-\overline y_\nu +\underline x^\intercal\vec S\underline y_\nu|$ 
in the integrand of the inner integral  in \eqref{eq:X2-integral}. Then we get the bound $O\big( \big(Ar_\nu (2^ns)^{-1} \big)^{-(2N-2m)}\big)$ 
for this inner integral. This proves the second bound in \eqref{eq:L1claim}.

Finally, we  prove the third bound in \eqref{eq:L1claim};  assume $2^n\ge r_\nu$ and extend the integration over all of $\bbR^{d+2m}$. Let,  with $\delta\in (0,1)$,  
\[ X_3=\{x:|\underline x-\underline y_\nu|\ge r_\nu\, (\frac{2^{n}}{r_\nu})^{1+\delta} \}, \quad X_4=X_3^\complement.\]
Clearly for $y\in B(y_\nu, r_\nu)$ we have $|\underline x-\underline y|\approx |\underline x-\underline y_\nu|$ and thus integrating first in $\ox$, 
\begin{align*} 
\int_{X_{3}} |b_\nu*\cK^{k,l}_{2^n s} (x)|dx\,&\lc \, 
2^{k(m+1)} \int_{|\underline x-\underline y_\nu|\ge r_\nu(2^{n}/r_\nu)^{1+\delta}  } \frac{(2^ns)^{-d} } { (\frac{|\underline x-\underline y_\nu|}{2^ns} )^N } 
d\underline x\,\|b_\nu\|_1
\\&\lc 2^{k(m+1)}(2^n r_\nu^{-1})^{-\delta(N-d)}  \|b_\nu\|_1.
\end{align*}
For $x\in X_4$ we use the mean value zero property of $b_\nu$ and write 
\begin{multline*} b_\nu*\cK^{k,l}_{2^ns} (x)= \\\int b_\nu(y) \Big[\cK^{k,l}_{2^n s} (\ux-\uy,\ox-\oy+\ux^\intercal\vec S \uy) 
- \cK^{k,l}_{2^n s} (\ux-\uy_\nu,\ox-\oy_\nu+\ux^\intercal\vec S \uy_\nu)\Big] dy
\end{multline*}
and from this 
\begin{multline*}
    \int_{X_{4}} |b_\nu*\cK^{k,l}_{2^n s} (x)|dx\lc 2^{k(m+2)} \|b_\nu\|_1
    \,\,  \times \\
\sup_{y\in B(y_\nu, r_\nu)} \Big[ (2^n s)^{-1}|\underline y_\nu-\underline y|   +(2^{2n}s^2)^{-1} 
\sup_{x\in X_4} 
|-\overline y+\overline y_\nu+\underline x^\intercal  \vec S(\underline y-\underline y_\nu) | \Big].
\end{multline*}

For $x\in X_4$ and $y\in B(y_\nu, r_\nu)$,
\begin{align*} |-\oy+\oy_\nu +\underline x^\intercal  \vec S(\underline y-\underline y_\nu)|&\le 
|(\underline x-\underline y_\nu)^\intercal  \vec S(\underline y-\underline y_\nu)|+ |
-\oy+\oy_\nu +\underline y_{\nu}^\intercal \vec S(y-y_\nu)| 
\\&\le |\underline x-\underline y_\nu|\|\vec S\||\underline y-\underline y_\nu|+
|\oy-\oy_\nu -\underline y_{\nu}^\intercal \vec Sy| 
\\ &\le 2r_\nu\big(\frac{2^n}{r_\nu}\big)^{1+\delta} \|\vec S\| r_\nu+ r_\nu^2 \lc   2^{n(1+\delta)} r_\nu^{1-\delta} + r_\nu^2.
\end{align*}
Using this in the estimate above and combining with the integral over $X_3$ yields 
\begin{multline*}
\int |b_\nu*\cK^{k,l}_{2^n s} (x)|dx \\ \lc \|b_\nu\|_1 
\Big[ 2^{k(m+1)} \big(\frac{r_\nu}{2^n}\big)^{\delta(N-d)} + 2^{k(m+2)} 
\Big( \frac{r_\nu}{2^n} +\big(\frac{r_\nu}{2^n}\big)^{1-\delta}+ \frac{r_\nu^2}{2^{2n}}\Big)\Big]
\end{multline*} 
and choosing $\delta=1/2$ gives the third estimate in \eqref{eq:L1claim}, for $2^n\ge r_\nu$. 
\end{proof}

\begin{proof}[The case $2\le p<\infty$] We prove   \eqref{eq:tobeproven3}. We are again using the Sobolev embedding inequality \eqref{Besov-embedding}, now for $p>2$. We proceed similarly as in the $p<2$ case (however the proof  is now  simpler since we are working with $\ell^p$-valued functions and not with $\ell^2$-valued functions and this allows to use  trivial $L^\infty$ bounds  in place of the previously used Calder\'on-Zygmund estimates). After  uses of H\"older's  inequality we get 
\begin{multline*} 
\Big\| \Big(\sum_{n\in \bbZ} \| u(s) f*\cK^{k,l}_{2^n s} \|_{B^{1/p}_{p,1}}^p\Big)^{\frac 1p} \Big\|_p \lc \Big(\sum_{n\in \bbZ} \int_I \big \|f*\cK^{k,l}_{2^ns} \big\|_p^p ds\Big)^{\frac 1p}
\\+ \Big(\sum_{n\in \bbZ} \int_I \big \|f*\cK^{k,l}_{2^ns} \big\|_p^p ds\Big)^{\frac{1}{pp'}}
\Big(\sum_{n\in \bbZ} \int_I \big \|f*\tfrac {d}{ds}\cK^{k,l}_{2^ns} \big\|_p^p ds\Big)^{\frac{1}{p^2}}\,.
\end{multline*}
Hence  \eqref{eq:tobeproven3} follows from
\begin{subequations} 
\Be\label{eq:tobeproven61}
\Big(\sum_{n\in \bbZ} \int_I \big \|f*\cK^{k,l}_{2^ns} \big\|_p^p ds\Big)^{\frac 1p}\lc_\eps(1+k)^{\frac 1p} 2^{-k\frac {d-1}{p}+l(\frac{d-2}{p}+\eps)}\|f\|_p
\Ee
and
\Be\label{eq:tobeproven62}
\Big(\sum_{n\in \bbZ} \int_I \big \|f*\tfrac {d}{ds}\cK^{k,l}_{2^ns} \big\|^p_p ds\Big)^{\frac 1p}\lc_\eps 2^{k-l} (1+k)^{\frac 1p} 2^{-k\frac {d-1}{p}+l(\frac{d-2}{p}+\eps)}\|f\|_p\,.
\Ee
\end{subequations}
We now have for $p=\infty$ the inequalities
\begin{subequations} \Be\label{eq:Linfty-one}
\sup_{n}\sup_{s\in I}\sup_{x\in G} |f*\cK^{k,l}_{2^ns}(x)| \lc \|f\|_\infty 
\Ee
and
\Be\label{eq:Linfty-two} \sup_{n}\sup_{s\in I}\sup_{x\in G}
|f*\tfrac {d}{ds}\cK^{k,l}_{2^ns} (x)|\lc 2^{k-l} \|f\|_\infty
\Ee
\end{subequations}
which are  immediate consequences of $\|\cK^{k,l}_{2^ns}\|= O(1)$, $\|\tfrac {d}{ds} 
\cK^{k,l}_{2^ns}\|= O(2^{k-l})$ (see \eqref{eq:L1bound}). Inequality \eqref{eq:tobeproven61} follows from 
\eqref{eq:tobeproven51}, \eqref{eq:Linfty-one} by interpolation, and likewise
\eqref{eq:tobeproven62} follows from 
\eqref{eq:tobeproven52}, \eqref{eq:Linfty-two}. 
This finishes the proof of  \eqref{eq:tobeproven3}. 
\end{proof} 
\subsubsection*{Comment on Remark \ref{rem:higher-s-derv}}  An examination of the proof above allows, for fixed $p$, inclusions of factors of $2^{(k-l)\beta}$ on the left hand sides of the inequalities \eqref{eq:tobeproven2} for $p\le 2$ and \eqref{eq:tobeproven3} for $p>2$.
Specifically we can have $\beta<\frac{d-1}{p'}-\frac 1p$  for $\frac{d}{d-1}<p\le 2$ and $\beta<\frac{d-2}{p}$ for $p>2$. This observation can be used to justify replacing the Besov space $B^{1/p}_{p,1}$ in the $s$ variable with $B^{\beta+1/p}_{p,1}$ in those ranges. For the cases where $\beta+1/p\ge 1$  one needs to use that also for $j>1$ the terms  $2^{(l-k)j} (\frac{d}{ds})^j\nu^{k,l}_{s}$  behave like $\nu^{k,l}_s$ (in particular this requires a straightforward  extension of calculations for $j=1$ at the end of \S\ref{sec:shear} below).

\section{Basic  considerations for the $L^2$ estimate in  Proposition \ref{prop:mainAkl}} \label{sec:propprel}


It  suffices to prove the proposition for functions that are supported in a small neighborhood of the origin of diameter $\ll\vr^2\ll 1$ since one can use a standard argument using a tiling via the group translations to reduce to the general case (for more details we refer to  \S2 of \cite{RoosSeegerSrivastava-imrn}).  
We follow  \cite{RoosSeegerSrivastava-imrn} to discuss further reductions which will simplify the forthcoming $L^2$ bounds.


\subsection{A shear transformation} \label{sec:shear}
When acting on functions $f$ supported in an  $\varrho^2$-ball centered at the origin we can rewrite $f*\nu^{k,l}_t=\fA^{k,l} f(x,t)$ where \[ \fA^{k,l} f(x,t) 
=\int K^{k,l}_t(x,y) f(y) dy\] and  $K^{k,l}_t$ is given by 
\begin{align} \notag K^{k,l}_t(x,y)&=t^{-d-2m} \nu^{k,l}(t^{-1}(\ux-\uy), t^{-2} (\ox-\oy+ \ux^\intercal\vec S \uy))\\ 
 \label{eq:Kklt-osc}
 &= \mathring a(x,t,y) \iint \zeta_{k,l}(\sigma,\tau)e^{i \frac{\sigma}t (x_1-y_1-t g( \tfrac{x'-y'}t))+i\inn{ \frac{\tau}{t^2} } { \overline x-\overline y + \underline x^\intercal \vec S  \underline y} }d\sigma d\tau 
 \end{align}
 with 
 \[ \mathring  a(x,t,y)= (2\pi)^{-m-1} t^{-d-2m}
\beta_1(\tfrac{x_1-y_1}t,\tfrac{\ox-\oy+ \underline x^\intercal \vec S\underline y}{t^2}) \beta_0(\tfrac{x'-y'}t) \chi_\vr(y)\] 
and $y\mapsto \chi_\vr(y)$ supported in a $\vr^2$ neighborhood of the origin.
Notice that $ \mathring a$ lives, where 
 $|t-1|\le \vr^2$,  $|x_1-1|\lc \vr^2$, $|x'|, |\ox|, |y|\lc \vr^2$.  Introducing the frequency variables $\vth=(\vth_1,\overline\vth)\in \bbR^{m+1}$, with  $\vth_1=2^{1-k}t^{-1}\sigma$, $\overline \vth_i=2^{1-k}  t^{-2} \tau_i$  we can rewrite the integral as
\begin{multline}\label{eq:Kklt}
K^{k,l}_t(x,y)= 2^{(k-1)(1+m)}  \mathring a(x,t,y)  \iint \zeta_1(2^lt\vth_1) \upsilon_1(t,\vth_1,\overline\vth) \,\,
\times \\
 e^{i2^{k-1}  ( \vth_1 (x_1-y_1-t g( \tfrac{x'-y'}t))+\inn{ \overline \vth}{\overline x-\overline y + \underline x^\intercal\vec S  \underline y} )} d\vartheta_1 d\overline\vartheta.
\end{multline} 
where we have abbreviated \[\upsilon_1(t,\vth_1,\overline\vth)= t^{1+2m} \zeta_1((t^2\vartheta_1^2+t^4|\overline \vartheta|^2)^{1/2}).\]
When $l=k$ we get a similar formula where $\zeta_1(2^kt\vth_1)$ is replaced with  $\zeta_0(2^kt\vth_1).$

Following  \cite{RoosSeegerSrivastava-imrn} we rewrite the phase and verify that  
\begin{align}\notag
    &\vth_1\big(x_1-y_1-t g( \tfrac{x'-y'}t) \big)+\sum_{i=1}^m  \overline \vth_i ( \overline x_i-\overline y_i + \underline x^\intercal S_i  \underline y) 
    \\ \label{eq:rewritephase} &\qquad=  \big (\vth_1-\sum_{i=1}^m \overline\vth_i  \ux^\intercal  S_i e_1 \big )  \big (x_1-y_1-tg (\tfrac{x'-y'}t) \big)  
    \\ \notag
    &\qquad +
    \sum_{i=1}^m \overline\vth_i\big(\ox_i+ x_1 \ux^\intercal S_i e_1-\oy_i +\ux^\intercal S_i P^\intercal y' -\ux^\intercal S_i e_1 t g(\tfrac{x'-y'}t)   \big ).
\end{align}

Setting $\theta_1= \vth_1-\sum_{i=1}^m \overline \vth_i \ux^\intercal S_i e_1$,
$\overline \theta_i=\overline \vth_i$, we can write the Schwartz kernel using the $(\theta_1,\overline \theta)$ frequency variables.
Define the phase function $\Psi$ by 
\begin{multline}\label{eq:Psidef} \Psi(x,t,y,\theta)=\\ \theta_1(x_1-y_1-tg(\tfrac{x'-y'}t)) +\sum_{i=1}^m\overline \theta_i (\overline x_i -\overline y_i +\ux^\intercal S_iP^\intercal y'- \ux^\intercal S_i e_1 tg(\tfrac{x'-y'}{t}) )\,.
\end{multline}
Making the  substitution  $(\vth_1,\overline\vth)= (\theta_1+\ux^\intercal S^{\overline\theta} e_1, \overline\theta)$, here using the notation  $S^{\overline \theta}=\sum_{i=1}^m \overline\theta_iS_i$ we obtain 
\begin{multline} \label{eq:Kklt-alt} K^{k,l}_t(x,y)= 2^{(k-1)(1+m)}\mathring  a(x,t,y) 
\iint e^{i 2^{k-1} \Psi(\underline x, \overline x+x_1 \underline x^\intercal \vec S e_1, t,y,\theta) } \,\times\\
\zeta_1(2^l t(\theta_1+\ux^\intercal S^{\overline\theta} e_1)) \upsilon_1(t,\theta_1+\ux^\intercal S^{\overline\theta} e_1, \overline\theta) d\theta_1 d\overline\theta;
\end{multline}
here note the nonlinear shear transformation
\[ (\ux,\ox)\mapsto (\ux,\ox +x_1\ux^\intercal \vec Se_1) 
\]  which is  present  in the phase function. 
It is now natural to  consider a variant $\cA^{k,l}$ which is related to $\fA^{k,l}$ via this   shear transformation. 
Let
\[ \cA^{k,l}f(x,t)\equiv \cA^{k,l}_t f(x) = \int \sK^{k,l}_t(x,y) f(y) dy \] 
where the Schwartz kernel is given by  
\begin{multline} \label{eq:cKklt}
\sK^{k,l}_t(x,y)= 2^{(k-1)(1+m)}   a(x,t,y) \iint e^{i2^{k-1}  \Psi(x,t,y,\theta) }   \times \\ \zeta(2^lt(\theta_1 +  \ux^\intercal S^{\overline \theta} e_1)) \upsilon(t,\theta_1+\ux^\intercal S^{\overline\theta} e_1, \overline\theta)  
  d\theta_1 d\overline\theta,
\end{multline}  
here $s\mapsto \zeta(s)$ is supported where $|s|\approx 1$, and we use the modification that for $k=l$ we replace   $\zeta(2^kt(\theta_1 +  \ux^\intercal S^{\overline \theta} e_1))$ with $\zeta_0(2^kt(\theta_1 +  \ux^\intercal S^{\overline \theta} e_1)) $. 
We are still assuming that $a$ is supported  where
\Be\notag\label{eq:a-support}
\supp(a)\subset \{ (x,t,y): \text{$|t-1|\le \vr^2$,  $|x_1-1|\lc \vr^2$, $|x'|, |\ox|, |y|\lc \vr^2$} \}.
\Ee
Notice that the nonlinear shear transformation  does not essentially change this support assumption since by the skew-symmetry of the $S_i$ 
we have $|x_1\ux^\intercal S_i e_1|\lc \vr^2$.

 With the choice of \[a(x,t,y)  = \mathring a(\ux, \ox-x_1\ux^\intercal \vec S e_1,t,y), \quad \zeta=\zeta_1, \quad \upsilon=\upsilon_1\] 
 we get for $l<k$
\Be\label{eq:shear}  \fA^{k,l} f(\ux, \ox,t) = \cA^{k,l}_t f(\underline x, \ox+ x_1\underline x^\intercal \vec Se_1, t)\Ee
(and the same with $\zeta=\zeta_0$ if $k=l$). 

We deduce the $L^2$-estimate in Proposition \ref{prop:mainAkl} from the following variant.
\begin{prop} \label{prop:mainAklvar}  Let  $\eps>0$.
  Then there exists a constant $C = C(\eps)>0$ such that 
 \Be\label{est:aklvar}
        \|\cA^{k,l} f\|_{L^2(\bbR^{d+m}\times [\frac 12,2])}\le C 2^{-\frac{k(d-1)}{2}} 2^{l(\frac{d-2}{2}+\eps)} \|f\|_{L^2(\bbR^{d+m})}, \Ee
        with $C$  bounded as $\zeta$, $\upsilon $, $ a$ are varying over bounded subsets of $C^\infty_c$ (with the above support assumptions). 
\end{prop}

\begin{proof}[Proof that Proposition \ref{prop:mainAklvar} implies Proposition \ref{prop:mainAkl}] By \eqref{eq:shear} Proposition \ref{prop:mainAklvar} immediately implies the first half of \eqref{est:akldakl}, by a change of variable. 
To prove the derivative bound in \eqref{est:akldakl} first observe
\[ \partial_t \Psi(\ux, \ox+x_1\ux^\intercal \vec Se_1,t,y,\theta) = \big(\theta_1+\ux^\intercal S^{\overline\theta} e_1\big) \big( \inn{\tfrac{x'-y'}t}{\nabla g(\tfrac{x'-y'}{t}) }- g(\tfrac{x'-y'}{t})\big)\,.\]
From \eqref{eq:Kklt-alt} we calculate that 
\[\partial_t \fA^{k,l} f(x,t)= \sum_{i=1,2,3} \fA^{k,l,[i]}_t f(x)+ 2^{k-l}  \fA^{k,l,[4]}_t f(x) 
 \]
where the Schwartz kernel of $ \fA^{k,l,[i]}_t$ is given by $K^{k,l,[i]}_t$, defined as in 
\eqref{eq:Kklt-alt} but with $\zeta$, $a$, $\upsilon$ replaced by $\zeta^{[i]}$, $ \mathring a^{[i]}$, $\upsilon^{[i]}$ for $i=1,2,3,4$, resp.,  with the following  definitions
 (for $l<k$) 
 \[ \zeta^{[1]}(s)= s\zeta_1'(s),\quad 
 \zeta^{[2]} (s)= \zeta^{[3]}(s)=\zeta_1(s),\quad 
 \zeta^{[4]}(s) = \frac {is}2\zeta_1(s),
 \]
 \[\upsilon^{[1]}=\upsilon^{[2]}=\upsilon^{[4]} =\upsilon_1,\quad \upsilon^{[3]}= \partial_t \upsilon_1,\] 
 \[\mathring a^{[1]} = t^{-1} \mathring a,  \quad \mathring a^{[2]} = \partial_t \mathring a,\quad 
  \mathring a^{[3]}= \mathring a, \]
  and \[\mathring a^{[4]}(x,t,y)= t^{-1} \mathring a(x,t,y) \big(\inn{\tfrac{x'-y'}{t}}{\nabla g(\tfrac{x'-y'}{t})}-g(\tfrac{x'-y'}{t} )\big).\]

For $l=k$ replace $\zeta_1$ by $\zeta_0$. These formulas show that the derivative bound in \eqref{est:akldakl} follows from Proposition \ref{prop:mainAklvar} as well, as we have  
\[  \fA^{k,l,[i]}_t f(\ux, \ox) = \cA^{k,l,[i]}_t f(\underline x, \ox+ x_1\underline x^\intercal \vec Se_1)\]
where, 
with $a^{[i]}(x,t,y) =\mathring a^{[i]}(\ux, \ox-x_1\ux^\intercal \vec Se_1,t,y) $, the operator 
$\cA^{k,l,[i]}_t$  has Schwartz kernel 
\begin{multline*} 
\sK^{k,l,[i]}_t(x,y)= 2^{(k-1)(1+m)}  a^{[i]}(x,t,y)  \iint e^{i2^{k-1}  \Psi(x,t,y,\theta) }   \times \\ \zeta^{[i]}(2^lt(\theta_1 +  \ux^\intercal S^{\overline \theta} e_1)) \upsilon^{[i]} (t,\theta_1+\ux^\intercal S^{\overline\theta} e_1, \overline\theta)    d\theta_1 d\overline\theta.
\end{multline*}   
Now we can use a change of variables and apply Proposition \ref{prop:mainAklvar} to $\cA^{k,l,[i]}_t$, 
for $i=1,2,3,4$ to complete the proof of Proposition \ref{est:aklvar}. 
\end{proof}

\subsection{A family of  oscillatory integral operators} \label{sec:ProofofProp} 
It remains to prove Proposition \ref{prop:mainAklvar}. We reduce it to a result on oscillatory integrals acting on functions on $\bbR^d$. 
 Here we write, $x=(x_1,x')$, $y=(y_1,y')$ for the vectors in $\bbR^d$, omitting the underbar. In what follows we are given a skew-symmetric $d\times d$ matrix $S$ and assume that its matrix norm satisfies \Be \label{eq:Snormbd} c_0\le \|S\|\le c_0^{-1}\Ee with $0<c_0\le 1$; in particular the rank of $S$ is at least $2$.

We define the phase function $\psi$ by
\Be\label{Psidef}
\psi(x,t,y)= y_1(x_1-tg(\tfrac{x'-y'}{t}) )+x^\intercal S (P^\intercal y'-tg(\tfrac{x'-y'}{t})e_1)
\Ee and set \Be\label{eq:sigmadef} \sigma(x',y_1)= y_1+(x')^\intercal  P  Se_1 .\Ee
The function $\zeta_1$ can be split as $\zeta_1=\zeta_1^++\zeta_1^-$ where $\supp(\zeta_1^+)\subset (\frac 12,2)$ and $\supp(\zeta_1^-)\subset (-2,-\frac 12)$. 

Setting $\la=2^{k-1}$ and letting $l\le k$ we define, for functions $f\in L^2(\bbR^d)$,
\Be \label{Tlaldef} T^{\la,l} f (x,t)= \int e^{i\la \psi(x,t,y)} \chi_l(x,t,y)f(y) dy
\Ee 
where \Be\label{defofchil}
\chi_l(x,t,y)= \begin{cases} \chi(x,t,y)\zeta(2^lt\sigma(x',y_1)), \,\, &l\le k-1
\\
\chi(x,t,y)\zeta_0(2^lt\sigma(x',y_1)),\,\, &  l= k.
\end{cases} 
\Ee
 Here  $\chi$ is  $C^\infty_c$-function supported where  
 $t\approx 1$, $|x'|, |y'|\le \varrho$, and the diameter of $\supp (\chi)$ does not exceed $ \varrho$. For $l\le k-1$ we use the convention for  $\zeta$  to be either  $\zeta_1^+$ or $\zeta_1^-$. Note then that for $l\le k-1$ we have  $|\sigma|\approx 2^{-l}$ on $\supp (\chi_l)$ and in addition the sign of $\sigma$ is the same for all $(x,t,y)$ in the support. 

\begin{prop} \label{prop:Tlal} Suppose $c_0\le \|S\|\le c_0^{-1} $. 
For $\eps>0$,
\Be \| T^{\la,l}f\|_{L^2(\bbR^{d}\times [1/2,2])}\lc C_\eps 2^{l(\frac{d-2}{2}+\eps)}\la^{-\frac{d}{2}} \|f\|_{L^2(\bbR^d)}\,.
\Ee
The constant $C_\eps$ depends on $c_0$ but not on the specific matrix $S$,  and stays bounded if $\zeta_0, \zeta_1^\pm, \chi$ 
range over a bounded set of $C^\infty_c$ functions. 
\end{prop} 
For $l\gg 1$ this is the key technical result of this paper; see \S\ref{sec:Tlal}.

\subsection{Reduction of Proposition \ref{prop:mainAklvar} to oscillatory integral operators} We will use Proposition \ref{prop:Tlal} to deduce Proposition \ref{prop:mainAklvar}. The estimate is more straightforward if $a$ can be written as a tensor product of functions of each of the variables $x_i, t, y_j$. 
To reduce to this situation we choose functions $x_i\mapsto \alpha_i(x_i)$, $t\mapsto \ga(t)$, $y_j\mapsto \beta_j(y_j)$, $1\le i,j \le d+m$,
all with compact support such that \[\breve a(x,t,y):=\ga(t)\prod_{i=1}^{d+m}\alpha _i(x_i)\prod_{j=1}^{d+m} \beta_j(y_j) \]  equals $1$ on $\supp(a)$, so that the support of each factor is contained in an interval of length less than $2\pi$. 

On the support of $\breve a$ we have the following  Fourier series expansion
\begin{align*}
a(x,t,y) =\sum_{ (n,\nu,\mu)\in \bbZ\times\bbZ^{d+m}\times \bbZ^{d+m}} c_{n,\nu,\mu} e^{itn  } \prod_{i=1}^{d+m} e^{ix_i\nu_i} \prod_{j=1}^{d+m} e^{iy_j\mu_j} 
\end{align*}
where the coefficients $c_{n,\nu,\mu} $ are rapidly decreasing.
This yields a decomposition \Be \label{eq:Fourierdec} \cA^{k,l} f(x,t)= \sum_{n,\nu,\mu} c_{n,\nu,\mu} e^{itn  } \prod_{i=1}^{d+m} e^{ix_i\nu_i}  \cA_{\mu}^{k,l} f(x,t)\Ee where $\cA_{\mu}^{k,l}$ is factorized  as a composition of three operators,
\Be \label{eq:factorization}\cA^{k,l}_\mu  f(x,t)= 2^{(k-1)(m+1)} \cF_k^{m} \cG^{k,l} \cF_{k,\mu}^{m+1} f(x,t);\Ee
here $\cF_k^m$ is defined on functions $(\ux,\overline\theta,t)\mapsto G(\ux,\overline\theta,t) $ by
\Be\label{eq:Fmdef} \cF_k^m G(\ux,\ox,t)= \prod_{i=d+1}^{d+m} \alpha_i(x_i)\int_{\bbR^{m} } G(\ux,\overline \theta,t) e^{i2^{k-1} \inn {\ox} {\overline\theta} }  d\overline\theta,
\Ee
$\cG^{k,l} $ is defined on functions $(\theta_1, y', \overline \theta)\mapsto F(\theta_1,y',\overline \theta) $ by
\begin{multline} \label{eq:cGkldef}
\cG^{k,l} F(\ux, \overline \theta,t)= \gamma(t) \prod_{i=1}^d \alpha_i(x_i) \int_{\theta_1,y'} 
e^{i2^{k-1} \psi^{\overline\theta} (x_1,x',t,\theta_1,y')} \\ \times\,
\zeta_1(2^lt(\theta_1 +  \ux^\intercal S^{\overline \theta} e_1))
\upsilon(t,\theta_1,\overline\theta)
\prod_{j=2}^d \beta_j(y_j) F(\theta_1,y',\overline\theta) 
d\theta_1dy' 
\end{multline}
with 
\Be\label{eq:psithetadef} 
\psi^{\overline\theta} (x_1, x', t, \theta_1, y')= 
\theta_1(x_1-tg(\tfrac{x'-y'}{t})) + \ux^\intercal S^{\overline\theta} (P^\intercal y'-tg(\tfrac{x'-y'}{t})e_1),
\Ee
and finally $\cF_{k,\mu}^{m+1}$ is defined on functions $(y_1,y', \overline y)\mapsto f(y_1,y',\overline y)$ by
\begin{multline} \label{eq:Fm+1mudef}
\cF_{k,\mu}^{m+1} f(\theta_1,y',\overline\theta) \\=\int e^{-i2^{k-1} (y_1\theta_1+\inn{\oy}{\overline\theta})} e^{i\inn{\mu}{y} } \beta_1(y_1) \prod_{j=d+1}^{d+m}  \beta_j(y_j) f(y_1,y', \overline y) dy_1d\overline y.
\end{multline}
We have the estimates
\begin{align} \label{cFmest}&\| \cF^m_k G\|_{L^2(\bbR^{d+m+1})} \lc 2^{-k m/2} \|G\|_{L^2(\bbR^{d+m+1})} 
\\
\label{cGklest}
&\| \cG^{k,l} F\|_{L^2(\bbR^{d+m+1})} \le C_\eps 2^{l( \frac{d-2}{2}+\eps)} 2^{-k d/2} \|F\|_{L^2(\bbR^{d+m})}
\\
&\label{cFm+1est}\|\cF^{m+1}_{k,\mu} f\|_{L^2(\bbR^{d+m})} \lc 2^{-k(m+1)/2} \|f\|_{L^2(\bbR^{d+m})}
\end{align}
and clearly the desired estimate \eqref{est:aklvar}  follows from 
\eqref{cFmest}, \eqref{cGklest}, \eqref{cFm+1est} in conjunction with 
\eqref{eq:Fourierdec}, \eqref{eq:factorization} and the rapid decay of the $c_{n,\mu,\nu}$.

We  justify the $L^2$ estimates. \eqref{cFmest} is an immediate consequence of Plancherel's theorem in $\bbR^m$ and likewise \eqref{cFm+1est} is a consequence of Plancherel's theorem in $\bbR^{m+1}$. It remains to consider \eqref{cGklest}; here we rely on Proposition \ref{prop:Tlal}.
With $\psi^{\overline\theta}$ as in \eqref{eq:psithetadef} define for functions 
$(\theta_1,y')\mapsto g(\theta_1,y')$ 
\begin{multline}  \cT_{\overline \theta}^{\la,l} g(\ux,t)= \\
\int_{\theta_1,y'} \exp({i\la \psi^{\overline\theta}(\ux,t,\theta_1,y')})  \chi^{\overline\theta}
(\ux,t,\theta_1,y') \zeta_1(2^lt \sigma^{\overline\theta}  (x', \theta_1) ) g(\theta_1,y') d\theta_1dy' 
\end{multline}
where $\sigma^{\overline\theta}  (x',\theta_1) = \theta_1+\ux^\intercal S^{\overline\theta} e_1 = \theta_1+ (x')^\intercal P S^{\overline\theta} e_1$; moreover
\[\chi^{\overline\theta}(\ux,t, \theta_1,y')= \gamma (t) \upsilon(t,\theta_1,\overline{\theta})\prod_{i=1}^{d} \alpha_i(x_i)\prod_{j=2}^d \beta_j(y_j) \,.
\]
By Proposition \ref{prop:Tlal} we have with  $\la\approx 2^k$
\Be  
\label{eq:cTlalL2}
\| \cT_{\overline \theta}^{\la,l} g\|_{L^2(\bbR^{d+1})} \lc 2^{l(\frac{d-2}{2}+\eps)} 2^{-kd/2} \|g\|_{L^2(\bbR^d)}\Ee 
uniformly in $\overline\theta$; 
note that we have exactly the setup in \eqref{Tlaldef}, except there we use the notation  $x$ for $\ux$, $y_1$ for $\theta_1$, and $S$ for $S^{\overline\theta}$. For the estimate \eqref{eq:cTlalL2} the uniformity assertion in Proposition \ref{prop:Tlal} is crucial and so is the assumption of the $S_i$ being linearly independent and therefore satisfying \eqref{eq:lincombofSi}. 
We have $\cG^{k,l} F(x,\overline\theta,t)_= \cT^{2^{k-1},l}_{\overline\theta}[ F(\cdot,\overline \theta)]$ and thus applying \eqref{eq:cTlalL2} gives \eqref{cGklest}. This covers the case $l<k$, and the case $l=k$ is analogous, requiring a minor notational modification. 
This finishes the proof of Proposition \ref{prop:mainAklvar}, given Proposition \ref{prop:Tlal}.  

\section{Proof of Proposition \ref{prop:Tlal}}\label{sec:Tlal}
For small $l$ we shall rely on a standard  $T^*T$ argument from \cite{Hormander1973}. 
The main part of the proof  concerns the case of large $l$; here we  rely on almost orthogonality arguments based  on the Cotlar-Stein lemma, in the following version. Consider a finite set $\sV$ indexing  bounded operators $T_\nu:H_1\to H_2$ where $H_1$, $H_2$ are Hilbert spaces. Then we have the following bound for the operator norm of the sum:
\Be \label{eq:CotlarStein}\Big\|\sum_{\nu \in \sV} T_\nu\Big\|_{H_1\to H_2} \lc  \sup_{\nu}  
 \sum_{\nu'} \|  T_\nu^* T_{\nu'}\|_{H_1\to H_1} ^{1/2} +\sup_{\nu} \sum_{\nu'} \|  T_\nu T^*_{\nu'}\|_{H_2\to H_2} ^{1/2}.
\Ee
This well known version  follows by a simple modification of the proof in \cite[ch. VII.2]{Stein-harmonic} (\cf. also \cite[p.223]{comech-survey}). 

\subsection{The case of small $l$} \label{sec:small-l} 
This is the regime where one can use a standard $T^*T$ argument (cf. 
\cite{Hormander1973}). 
Recall that $g(0) = 1$, $\nabla g(0) = 0$, $\diam(\supp(\chi)) \le \vr\ll 1$, in particular $|x'|, |y'|\le \varrho\ll1 $ for $(x,t,y)\in \supp(\chi)$.  
Denote as before \Be \label{eq:newsigma} \sigma = \sigma(x',y_1) = y_1 + (x')^\intercal PS e_1.
\Ee
Let the $(d+1)\times d$ matrix $\partial_y^\intercal\partial_{x,t} \psi$
be defined by 
$(\partial_y^\intercal\partial_{x,t} \psi)_{i,j}= \partial_{x_i} \partial_{ y_j}\psi$ for $1\le i,j\le d$ and 
$(\partial_y^\intercal\partial_{x,t} \psi)_{d+1,j}= \partial_{t} \partial_{y_j}\psi$ for $1\le j\le d$. One calculates (\cite{RoosSeegerSrivastava-imrn})
\begin{multline}\label{eq:mixhessian}
\partial_y^\intercal \partial_{x,t}  \psi\big|_{(x,t,y)}=\\
     \begin{pmatrix}
    1 & e_1^\intercal S P^\intercal \\
    -g'(\frac{x'-y'}{t}) & t^{-1}\sigma(x',y_1) g''(\frac{x'-y'}{t}) + PS P^\intercal + PSe_1(g'(\frac{x'-y'}{t}))^\intercal \\
    -1+\wt g(x,t,y) & -t^{-2}\sigma (x',y_1)(x'-y')^\intercal g''(\frac{x'-y'}{t})
    \end{pmatrix}
\end{multline}
where 
$
\wt g(x,t,y) := 1-g(\frac{x'-y'}{t})+t^{-1}g'(\frac{x'-y'}{t})(x'-y').
$

Using
\eqref{eq:mixhessian} we obtain for the determinant of the $d\times d$ submatrix $\partial_y^\intercal \partial_{x}  \psi$
\[
\det(\partial_y^\intercal \partial_x \psi(x,t,y)) = \det\big(t^{-1}\sigma (x',y_1)g''(\tfrac{x'-y'}{t}) + PS P^\intercal\big) + O(\vr).
\]
From \cite[Lemma 5.3]{MuellerSeeger2004}, it follows that the matrix $t^{-1}\sigma g''(\frac{x'-y'}{t}) + PS P^\intercal$ is invertible. This says that $\partial_y^\intercal \partial_x \psi(x,t,y) $ is invertible for all $(x,t,y) \in \supp(\chi)$. Also, the derivatives of the amplitude 
$\chi (x,t,y)\zeta(2^l t\sigma(x',y_1))$ are bounded when $2^l\approx 1$.  
Thus the standard oscillatory integral theorem from  \cite{Hormander1973} applies and we may conclude the bound 
$\|T^{\la,l} f(\cdot, t) \|_2\le C(l)  \la^{-d/2}\|f\|_2$  which one uses for $2^l\lc \vr^{-1}$.


\subsection{The case of large $l$}
\label{sec:Se1large}  We may assume  that $2^{-l}\ll\vr\ll 1$
(recall from  the beginning of \S\ref{sec:dyadic} the  specifications of the parameter $\vr$).
Choose an orthonormal basis $\fe_1,\dots, \fe_d$ with $\fe_1=e_1$, and  $Se_1 \in \text{span}(\fe_2)$.
Set 
\Be\label{deltadef} \delta_l=\max\{|Se_1|, 2^{-l}\} \Ee
To prepare for almost orthogonality arguments we tile $\bbR^d$ into boxes  with  sidelengths  $(2^{l\eps/d}\la^{-1}, 2^{l\eps/d} \la^{-1} \delta_l^{-1}, 2^{l(1+\eps/d)}\la^{-1}, \dots, 2^{l(1+\eps/d)}\la^{-1} )$, with the sides parallel to the $\fe_1,\fe_2, \fe_3,  \dots, \fe_d$.
The family of boxes   $\fs$ can be parametrized by 
$ \bbZ^d$; we define the lower corners by $ c(\fz)= 2^{l\eps/d} \big( \la^{-1} \fz_1 \fe_1 + \la^{-1}\delta_l^{-1} \fz_2 \fe_2+\sum_{i=3}^d 2^l\la^{-1}\fz_i \fe_i\big)$, 
and let 
\[\fs(\fz)= \{y: \inn{c(\fz_1,\dots, \fz_d) }{\fe_i}  \le \inn{y}{\fe_i}<   \inn{c(\fz_1+1,\dots, \fz_d+1) }{\fe_i} ,\, i=1,\dots, d\}.\]
We also write $\fc_\fs=c(\fz)$  if $\fs=\fs(\fz)$.
Denote by $\fS$ the (finite) family of those boxes which intersect $\{y:(x,t,y)\in \supp (\chi) \text{ for some } (x,t)\} $.
We then decompose 
\Be T^{\la,l}= \sum_{\fs\in \fS} T^{\la,l}_\fs, \text{ with } T^{\la,l}_\fs[f]= T^{\la,l}[f\bbone_{\fs}].
\Ee Note that 
\Be \label{eq:spatialorth} T^{\la,l}_\fs (T^{\la,l}_{\fs'})^*=0 \text{ if } \fs\neq \fs'.
\Ee

Notice that  we have 
$|T^{\la,l}_\fs f(x,t)|\lc |\fs|^{1/2} \|f\|_2$.
Because of the compact  support of the kernel in the $(x,t)$ variable we see that the $L^2$ operator norm   $\|T^{\la,l}_\fs\|_{2\to 2}$ is $O(|\fs|^{1/2}).$
 It is crucial for our analysis that this can be improved by a factor of $\delta_l^{1/2}$:
\begin{lem}\label{lem:gsll}
    There exists a constant $C>0$ independent of $\fs\in\fS$ such that the estimate
    \begin{align*}
        \big\|T^{\lambda,l}_\fs f(\cdot, t) \big\|_{L^2(\bbR^d)} \le C\lambda^{-\frac d2}2^{l(\frac{d-2}{2}+\eps)}  \|f\|_{L^2(\bbR^d)}
    \end{align*}
    holds for every $\fs\in\fS$, with $C$ independent of $t\in [\frac 14,4]$ and $\fs$.
\end{lem}

\begin{proof}[Proof of Lemma \ref{lem:gsll}]
We freeze $t\in[\frac 14,4]$  for this proof and write
$\cT^{\la,l}_\fs f(x)= T_\fs^{\la,l} f(x,t)$, all estimates will be uniform in $t$.

 We have $|\fs|\lc 2^{l(d-2+\eps)} \delta_l^{-1} \la^{-d} $ and therefore obtain from the Cauchy-Schwarz inequality 
\[\|\cT_\fs^{\la,l}\|_{L^2\to L^2} \lc 2^{l(d-2+\eps)/2} \delta_l^{-1/2} \la^{-d/2} .\]
Let $c_1\ll c_0$ be a small constant, and the displayed estimate is already sufficient if $|Se_1|\ge c_1$.
In what follows we consider the case $|Se_1|\le c_1$. We note that in this case 
\Be\label{b:pspint}
    \sup_{w'\in\R^{d-1},  |w'|=1 } |PSP^\intercal w'| \ge c_0/2.
\Ee
Indeed,  write $w=(w_1,w')$ and 
$Sw = \big(e_1^\intercal SP^\intercal w', w_1 PS e_1 + PSP^\intercal w'\big)$; we have $|e_1^\intercal SP^\intercal w'|+|w_1 PS e_1|\lc c_1|w| $ with 
$c_1\ll c_0$ and \eqref{b:pspint} holds by  \eqref{eq:Snormbd}.

 Let $d_\circ$ be the smallest integer greater than or equal to $(d-1)/2$. Since $PSP^\intercal$ is skew-symmetric, there exists nonnegative numbers $s_1\ge \dots \ge s_{d_\circ}$ and orthonormal vectors $\vec u_1,\dots \vec u_{d-1} \in \bbR^{d-1}$ such that
 \begin{subequations} \label{eq:ONB-proj}
\begin{align}
    \begin{aligned}
        &PSP^\intercal \vec u_{2i-1} = s_i \vec u_{2i},   \\
       &PSP^\intercal \vec u_{2i} = - s_i \vec u_{2i-1}, 
    \end{aligned} \quad 
\end{align}
 for $1\le i\le d_\circ$ if $d-1$ is even,  and
\begin{align}
    \begin{aligned}
        &PSP^\intercal \vec u_{2i-1} = s_i \vec u_{2i},   \\
       &PSP^\intercal \vec u_{2i} = - s_i \vec u_{2i-1}, 
    \end{aligned} \quad  PSP^\intercal \vec u_{2d_\circ-1} = 0
\end{align}
\end{subequations}
 for $1\le i\le d_\circ -1$, if $d-1$ is odd.
By \eqref{b:pspint}, we have $s_1\gc c_0$.

To estimate $\mathcal T_\fs^{\lambda, l}f$, we further decompose the slab $\fs$ into smaller pieces.  
We may write $PSe_1=\sum_{i=1}^{d-1} \alpha_i \vec u_i$ and let $ b= \beta_1\vec u_1+\beta_2\vec u_2$ where 
$\beta_1^2+\beta_2^2=1$ and $\alpha_2\beta_1-\alpha_1\beta_2=0$. Then $b$ is a unit vector in $\text{span}(\vec u_1,\vec u_2)$ with the property that
$PSP^\intercal b=-\beta_2s_1\vec u_1+\beta_1s_1\vec u_2$ is perpendicular to $PSe_1$. For later use notice that
$|PSP^\intercal b| =s_1.$

We now decompose $\fs$ into subsets $\fr_n(\fs)$ defined for $n\in \bbZ$ by 
\Be\label{eq:rnsdef}  \fr_n(\fs)=\{y=(y_1,y')\in \fs: 2^{l\eps/d}\la^{-1}n\le \inn{b}{y'} < 2^{l\eps/d}\la^{-1} (n+1)\}.\Ee
Define $\cT^{\la,l}_{\fs,n} f= \cT^{\la,l}_\fs[f\bbone_{\fr_n(\fs)}]$
so that 
$\cT^{\la,l}_\fs= \sum_n \cT^{\la,l}_{\fs,n}. $
As $\inn{P^\intercal b}{e_1}=0$   we have \[ |\fr_n(\fs)|\lc  2^{l(d-2+\eps)} \la^{-d} , \]
and by the Cauchy-Schwarz inequality we get
\Be \|\cT^{\la,l}_{\fs,n}\|_{L^2\to L^2}  
\lc 2^{l(d-2+\eps)/2}\la^{-d/2}.\Ee Since in view of the disjointness of the sets $\fr_n(\fs)$ we have 
$\cT^{\la,l}_{\fs, n} (\cT^{\la,l}_{\fs,n'})^*=0$ for $n\neq n'$ it suffices,  by the Cotlar-Stein Lemma,  to show 
\Be\label{eq:Tsn-orth} \| (\cT^{\la,l}_{\fs, n} )^* \cT^{\la,l}_{\fs,n'}\|_{L^2\to L^2} 
\lc 2^{l(d-2+\eps)}  \la^{-d} |n-n'|^{-N} \quad \text{ if } |n-n'|\ge C_1
\Ee for some large $C_1$. 

We now assume that  $y\in \fr_n(\fs)$, $z\in \fr_{n'}(\fs)$;  since both $y,z$ belong to $\fs$ this means that $|n-n'|\le C2^l$.
The Schwartz kernel of $(\cT^{\la,l}_{\fs,n})^*\cT^{\la,l}_{\fs,n'}$ is given by 
\Be\label{eq:Hnn'}H_{n,n'}(y,z)= \bbone_{\fr_n(\fs)}(y) \bbone_{\fr_{n'}(\fs)} \int e^{-i\la\phi(x,t,y,z)} \overline{\chi_l(x,t,y)}\chi_l(x,t,y) \,dx \Ee
where
\Be \label{eq:phidef} \phi(x,t,y,z)= \psi(x,t,y)-\psi(x,t,z).\Ee 
The argument will rely on an integration by parts using the directional derivative
\Be \label{eq-dirderivb}\inn{v}{\partial_{x'}} = \sum_{i=2}^d v_i \frac{\partial}{\partial x_i} \text{ with } v= \frac{PSP^\intercal b}{|PSP^\intercal b|}.\Ee
Note that 
\begin{align}  \label{eq:vderphi}
\inn{v}{\partial_{x'}}\phi(x,t,y,z)= \sum_{i=2}^d v_i\int_0^1 \partial_y^\intercal \partial_{x_i} \psi(x,t, w^\tau(y,z)) \,d\tau\, (y-z)&
\\
\label{eq:wtau} \text{ where }  w^\tau\equiv w^\tau(y,z):=(1-\tau)y+\tau z.&
\end{align}

Using \eqref{eq:mixhessian}, we write
\begin{multline}
    \label{eq:mixedHessianx'} \partial_y^\intercal \partial_{x'}\psi\big|_{(x,t,w^\tau)}  (y-z) = \inn{y'-z'}{b}PSP^\intercal b  + PSP^\intercal \Pi_{b^\perp}(y'-z')\,+ \\
     t^{-1}\sigma (x',w_1^\tau)g''(\tfrac{x'-{w^\tau}'}{t})(y'-z') - g'(\tfrac{x'-{w^\tau}'}{t})(y_1-z_1)  + PSe_1 (g'(\tfrac{x'-{w^\tau}'}{t}))^\intercal (y'-z').
\end{multline}
Since  $PSP^\intercal b$ and thus $v$ is perpendicular to both $PSe_1$, $PSP^\intercal \Pi_{b^\perp}(y'-z')$, and $|PSP^\intercal b| = s_1$ we have 
\begin{multline}
\label{eq:psibsec}
    \partial_y^\intercal \inn{v}{\partial_{x'}}\psi\big|_{(x,t,w^\tau)} (y-z) = s_1\inn{y'-z'}{b} +\\(ts_1)^{-1}\sigma(x', w_1^\tau) (PSP^\intercal b)^\intercal g''(\tfrac{x'-{w^{\tau}}'}{t})(y'-z')  - s_1^{-1}(y_1-z_1)(PSP^\intercal b)^\intercal g'(\tfrac{x'-{w^{\tau}}'}{t}).
\end{multline} 
Notice from \eqref{eq:newsigma} that 
\Be \label{eq:sigma-taudec}
\sigma(x',(1-\tau)y_1+\tau z_1)=(1-\tau)\sigma(x',y_1)+\tau\sigma(x',z_1). \Ee
Thus if
$\chi_l(x,t,y)\neq 0$ and $\chi_l(x,t,z)\neq 0$  then $\sigma (x', w_1^\tau(y,z)) =O( 2^{-l})$. 
Since $y,z \in \fs$, we also have 
    $|y_1-z_1|\lesssim \lambda^{-1}2^{l\eps/d} $ and $ |y'-z'|\lesssim \lambda^{-1}2^{(1+\eps/d)l}.$ 
Hence 
the expression in the second line of display \eqref{eq:psibsec} is $O(\la^{-1} 2^{l\eps/d}).$
Finally  $|\inn{y'-z'}{b}|\approx |n-n'| \lambda^{-1} 2^{l\eps/d}$  because $(y,z)\in \fr_n(\fs)\times \fr_{n'} (\fs)$. Thus,  we may use these observations in  \eqref{eq:vderphi}, \eqref{eq:psibsec} to conclude that
\Be\label{eq:dirderivphilowerbd}
\big|\inn{v}{\partial_{x'}} \phi(x,t,y,z)\big|\gc |n-n'| \la^{-1} 2^{l\eps/d}, \quad  \text{ if $|n-n'|\ge C_1$ }
\Ee
for a large constant $C_1$.
This lower bound allows us to integrate by parts in the integral \eqref{eq:Hnn'}.

Let $\cL$ be the formal adjoint  of 
$g\mapsto (-\inn{v}{\partial_{x'} }\phi)^{-1} \inn{v}{\partial_{x'}} g$, i.e. \[\cL g= \inn{v}{\partial_{x'}}\big( \frac{g}{ \inn{v}{\partial_{x'}} \phi} \big) 
=\frac{\inn{v}{\partial_{x'}} g}{ \inn{v}{\partial_{x'}} \phi} - \frac {g \inn{v}{\partial_{x'}}^2 \phi }
{ (\inn{v}{\partial_{x'} }\phi)^2}\,.\] Setting  
\Be\label{eq:etaldef} \eta_l(x,t,y,z):=
\overline{\chi_l(x,t,y)}\chi_l(x,t,z)
\Ee we have 
\Be\label{eq:Hnn'mod}H_{n,n'}(y,z)= \bbone_{\fr_n(\fs)}(y) \bbone_{\fr_{n'}(\fs)}(z)   \int e^{i\la\phi(x,t,y,z)} \frac{\cL^N\eta_l(x,t,y,z)}{ (-i\la)^N}\,dx.  \Ee

In order to estimate $\cL^N \eta_l$ we first observe that because $v$ and $PSe_1$ are perpendicular we have $\inn{v}{\partial_{x'}} \sigma(x',y_1)\equiv 0$ and  $\inn{v}{\partial_{x'} }\partial_{y} \sigma(x',y_1)\equiv 0$.
This implies that  the functions $\inn{v}{\partial_{x'}}^j\partial_{y_i} \psi(x,t,w^\tau) $, $2\le i\le d$, $j\ge 2$,   
belong to ideal generated by $\sigma(x',y_1)$, a quantity which is $O(2^{-l})$. This in turn implies that for $(x,t,y,z)\in \supp(\eta_l)$, $y,z\in \fs$ 
\[|\inn {v}{\partial_{x'}}^j \phi(x,t,y,z)|  \lc |y_1-z_1|+2^{-l} |y'-z'|\lc 2^{l\eps/d}\la^{-1}.\]
A straightforward calculation together with \eqref{eq:dirderivphilowerbd} shows  
\[ |\cL^N \eta_l(x,t,y,z) | \lc \la^N (2^{l\eps/d}|n-n'|)^{-N} \text{ for $y\in \fr_n(\fs)$, $z\in \fr_{n'}(\fs)$} \] and from \eqref{eq:Hnn'mod} 
we get \[\sup_z\int|H_{n,n'}(y,z) |\,dy+\sup_y\int |H_{n,n'}(y,z) |dz \lc 2^{l(d-2 +\eps -N\eps/d)} \la^{-d} |n-n'|^{-N} \]
for $|n-n'|\ge C_1$. Hence we get \eqref{eq:Tsn-orth} by Schur's test.
\end{proof}

In order to finish the proof of Proposition \ref{prop:Tlal}   using Lemma \ref{lem:gsll} and \eqref{eq:CotlarStein},  it remains to   show that the operator norms of $(T^{\la,l}_\fs)^* T^{\la,l}_{\fs'} $ are small if $\fs$, $\fs'$ are far apart.  In order to quantify this we decompose the set of pairs $(\fs, \fs')$ in families $\cU_{\ka_1,\ka_2,\ka_3}$ with $\ka_i\in \{0,1,2,\dots\}$
which we now define. 
For $\fs\in\fS$, we write 
$c_\fs^i= \inn{c_\fs}{\fe_i},$ $i=1,2$, $c_\fs^\perp = \Pi_{\text{span}(\fe_1, \fe_2)^\perp}=\sum_{k=3} ^dc_\fs^k \fe_k$.

Let $\ka_1, \ka_2, \ka_3\in \bbN_0\equiv\{0,1,2,\dots\}$  such that 
$2^{\ka_i}\le  4\la$.  To parse the following definition note  that $2\floor{2^{\ka-1}}= 2^\ka$ if $\ka\in \bbN$ and  $2\floor{2^{\ka-1}}= 0$  if $\ka=0$.
We  define $\mathcal U_{\ka_1,\ka_2,\ka_3} $ as the set of pairs $(\fs,\fs')\in \fS\times\fS$ such that 
\begin{subequations} \label{eq:UKdef}
\begin{align}
   &\quad 2\floor{ 2^{\ka_1-1} }\lambda^{-1}\le 2^{-l\eps/d}|c_\fs^1 - c_{\fs'}^1|\le 2^{\ka_1+1}\lambda^{-1},\\
    &\quad 2\floor{ 2^{\ka_2-1} }2^{\ka_1} \lambda^{-1} \delta_l^{-1}\le 2^{-l\eps/d}|c_\fs^2-c_{\fs'}^2|<2^{\ka_2+\ka_1+1} \delta_l^{-1} \lambda^{-1}, \\
    &\quad 2\floor{ 2^{\ka_3-1}} 2^{\ka_2+\ka_1} \lambda^{-1}2^l\le 2^{-l\eps/d}|c_\fs^\perp - c_{\fs'}^\perp|\le2^{\ka_3+\ka_2+\ka_1+1} \lambda^{-1}2^l\,.
\end{align} \end{subequations}  
We let $\cU^\fs_{\ka_1,\ka_2,\ka_3}=\{\fs'\in \fS: (\fs,\fs')\in \cU_{\ka_1,\ka_2,\ka_3}\}$. 
It is easy to see that for every $\fs\in \fS$
\[ \fS= \bigcup_{\ka_1,\ka_2,\ka_3\ge 0} \cU_{\ka_1,\ka_2,\ka_3}^\fs.\]  When all $\ka_i$ are small we can use Lemma \ref{lem:gsll}. The following lemma gives improved bounds if at least one of $\ka_1,\ka_2,\ka_3$ is large.

\begin{lem} \label{lem:almostortho} For $\ka_1,\ka_2,\ka_3\in \bbN_0$, $(\fs,\fs')\in \fS\times\fS$ we have the following estimates:


(i) If $\ka_1\ge 5$, $\ka_2, \ka_3\le 10 $  and $(\fs,\fs')\in \cU_{\ka_1,\ka_2,\ka_3}$ then for all $N>0$
\Be \label{eq:ka1large}
\|(T^{\la,l}_\fs)^* T^{\la,l}_{\fs'}\|_{L^2\to L^2} \lc_N 2^{-(\frac{l\eps}d+\ka_1)N } 2^{l(d-2+\eps)} \delta_l^{-1} \la^{-d} \,.
\Ee

(ii) If $\ka_2\ge 5$, $\ka_3\le 10$ and $(\fs,\fs')\in \cU_{\ka_1,\ka_2,\ka_3}$ then for all $N>0$
\Be \label{eq:ka2large} \|(T^{\la,l}_\fs)^* T^{\la,l}_{\fs'}\|_{L^2\to L^2} \lc_N  2^{-(\frac{l\eps}d+\ka_1+\ka_2)N } 2^{l(d-2+\eps)} \delta_l^{-1} \la^{-d}\,.
\Ee

(iii) If $\ka_3\ge 5$ and $(\fs,\fs')\in \cU_{\ka_1,\ka_2,\ka_3}$ then for all $N>0$
\Be \label{eq:ka3large} \|(T^{\la,l}_\fs)^* T^{\la,l}_{\fs'}\|_{L^2\to L^2} \lc_N  2^{-(\frac{l\eps}d+\ka_1+\ka_2+\ka_3)N } 2^{l(d-2+\eps)} \delta_l^{-1} \la^{-d}\,.
\Ee
\end{lem}

Lemma \ref{lem:almostortho} will be proved in \S\ref{sec:almostortho}. In each case, we will analyze for $y\in \fs$ and $z\in\fs'$ the size of the Schwartz kernel $\sK_{\fs,\fs'}\equiv \sK_{\fs,\fs'}^{\la,l}$  of 
 $(T^{\la,l}_\fs)^* T^{\la,l}_{\fs'} $ given by  
\Be\label{eq:Kss'}\sK_{\fs,\fs'} (y,z)= \bbone_{\fs}(y) \bbone_{\fs'}(z) \int e^{i\la\phi(x,t,y,z)} \eta_l(x,t,y,z) 
\,dx dt \Ee
with $\eta_l$ as in \eqref{eq:etaldef}. Note that  whenever $l\le k-1$ the definition  \eqref{eq:etaldef} of $\eta_l$ via \eqref{defofchil} implies that $\sigma(x',y_1)$ and $\sigma(x',z_1)$ have the same sign,  and absolute value comparable to $2^{-l}$.
Our  proof will then rely  on various integration by parts in the integral \eqref{eq:Kss'}. Specifically for $(\fs,\fs')\in \cU_{\ka_1,\ka_2,\ka_3}$ we use 
integration by parts with respect to $t$, when $\ka_1\ge 5$, $\ka_2,\kappa_3\le 10$, integration by parts with respect to $x_1$, when  $|Se_1|\ge 2^{-l}$ and $\kappa_2\ge 5$, $\kappa_3\le 10$, and integration by parts using the directional derivative $\inn{\frac{y'-z'}{|y'-z'|}}{\partial_{x'}} $,   either when $\ka_3\ge 5$ or when  
      $\ka_2\ge 5$,  $\ka_3\le 10$,  $|Se_1|\le 2^{-l}$ (see \S\ref{sec:almostortho} below). 
Assuming Lemma \ref{lem:almostortho}  we can now give the 
\begin{proof}[Proof of Proposition \ref{prop:Tlal}]
We verify \eqref{eq:CotlarStein}. In view of \eqref{eq:spatialorth} it suffices to prove for each $\fs$ 
\Be\label{eq:sumT*T} \sum_{\fs'} \|  (T^{\la,l}_\fs)^* T^{\la,l}_{\fs'}\| ^{1/2}
\lc 2^{l (\frac{d-2+\eps}2)}\la^{-d/2} \Ee
with implicit constant independent of $\fs$. 
From  the definition of $\cU_{\ka_1,\ka_2,\ka_3} $ 
it is  easy to see that
\Be\label{eq:U-cardinality}
\sup_\fs \#(\cU_{\ka_1,\ka_2,\ka_3}^\fs)\lc 
2^{\ka_1d+\ka_2(d-1)+\ka_3(d-2)}.
\Ee
From Lemma \ref{lem:gsll} 
we have \[\|(T^{\la,l}_\fs)^* T^{\la,l}_{\fs'}\| \lc 
\|T^{\la,l}_\fs\| \| T^{\la,l}_{\fs'}\|
\lc 2^{l(d-2+\eps)}\la^{-d} \]  and thus by  \eqref{eq:U-cardinality} for $\ka_i\le 10$, $i=1,2,3$ we have 
\Be\label{eq:U111}\sup_\fs \sum_{\ka_1,\ka_2,\ka_3\le 10} \sum_{\fs'\in \cU_{\ka_1,\ka_2,\ka_3}^\fs } \|(T^{\la,l}_\fs)^* T^{\la,l}_{\fs'}\|^{1/2} \lc 2^{l(d-2+\eps)/2}\la^{-d/2}. \Ee
Moreover using that $\delta_l^{-1}\le 2^l$ we obtain from 
Lemma \ref{lem:almostortho} and \eqref{eq:U-cardinality} 
\begin{multline*}
\sup_\fs \sum_{\max\{\ka_1,\ka_2,\ka_3\}\ge 5}  \sum_{\fs'\in \cU_{\ka_1,\ka_2,\ka_3}^\fs } \|(T^{\la,l}_\fs)^* T^{\la,l}_{\fs'}\|^{1/2} \\ \lc 2^{l\frac{d-2+\eps}2}\la^{-\frac d2} 2^{l(\frac {1} 2- \frac{\eps N}{2d} )} \sum_{(\ka_1,\ka_2,\ka_3)\in \bbN_0^3} 2^{-(\ka_1+\ka_2+\ka_3) (\frac N2-d)}.
\end{multline*} 
For $N>2d$ we can sum in $\ka_1,\ka_2,\ka_3$, and if in addition  we also choose $N>1+d/\eps$ we get the  bound $O(2^{l\frac{d-2+\eps}2}\la^{-\frac d2} )$ for the last display and \eqref{eq:sumT*T}  follows.
\end{proof}

\subsection{Proof of Lemma \ref{lem:almostortho}}\label{sec:almostortho} We verify first 
\eqref{eq:ka1large}, then \eqref{eq:ka2large} in the case $|Se_1|\ge 2^{-l}$ and then give a unified treatment 
of \eqref{eq:ka3large} and  the case $|Se_1|\le 2^{-l}$ in \eqref{eq:ka2large}.
\subsubsection*{Proof of \eqref{eq:ka1large}} We are now in the case $\ka_1\ge 5$, and $\ka_2,\ka_3\le 10$   in \eqref{eq:UKdef}.

We  examine  the Schwartz kernel $\sK_{\fs,\fs'}$ of $(T^{\la,l}_{\fs})^*T^{\la,l}_{\fs'}$ given in  \eqref{eq:Kss'}; in the case under consideration we have
$|y_1-z_1|\approx 2^{\ka_1}\la^{-1} 2^{l\eps/d}$, $|\inn{y-z}{\fe_2}| \lc \delta_l^{-1} \la^{-1}2^{\ka_1}  2^{l\eps/d}  $, $|\inn{y-z}{\fe_i} |\lc \la^{-1}2^{\ka_1} 2^{l(1+\eps/d)},$ $i=3,\dots, d$.
We now integrate  by parts with respect to $t$; for this observe  that (with $w^\tau$ as in \eqref{eq:wtau})
\Be\label{eq:tphiexpr}\begin{aligned}
&\partial_t\phi(x,t,y,z)=\partial_t\psi(x,t,y)-\partial_t\psi(x,t,z)
   = -(y_1-z_1) +\\ &\int_0^1 \big[ \widetilde g(x,t,w^\tau) (y_1-z_1)-t^{-2}\sigma(x',w_1^\tau) (x'-{w^\tau}')^\intercal g''(\tfrac{x'-{w^\tau}' } {t})(y'-z') \big]d\tau.
\end{aligned}
\Ee
Since $|x'-y'|\le\vr\ll 1$ we have $|\widetilde g(x,t,w^\tau)|\ll 1$, and from, 
\eqref{eq:tphiexpr} 
and $|\sigma|\lc 2^{-l}$ we see that 
\Be\label{eq:phitlower} |\partial_{t}\phi(x,t,y,z)| \approx |y_1-z_1|\approx  2^{\ka_1} \la^{-1} 2^{l\eps/d} \,.
\Ee

Observe that the higher $t$-derivatives of $\widetilde g$ are $\lc \vr\ll1$. Moreover  $\sigma$ does not depend on $t$ and we see that
\begin{align*} 
&|\partial^N_t\!\phi(x,t,y,z)|\lc_N |y_1-z_1|+ 2^{-l} |y'-z'| \lc 2^{\ka_1} 2^{l\eps/d}\la^{-1},
\\
&|\partial_t^N [\eta_l(x,t,y,z)]|\lc_N 1.
\end{align*}
Hence integration by parts with respect to $t$ yields  the pointwise bound
$|\sK_{\fs,\fs'}(y,z)|\lc (2^{\ka_1} 2^{l\eps/d})^{-N}$ which gives
\[ \sup_y\int |\sK_{\fs,\fs'} (y,z) |dz+ \sup_z \int|\sK_{\fs,\fs'} (y,z)| dy  \lc_N
\frac{\la^{-d} 2^{l(d-2+\eps)}\delta_l^{-1} }{( 2^{\ka_1} 2^{l\eps/d})^{N} }.
\] As  $\delta_l^{-1}\le 2^l$ we obtain   \eqref{eq:ka1large}, by Schur's test.

\subsubsection*{Proof of \eqref{eq:ka2large} in the case $|Se_1|\ge 2^{-l}$}
This now concerns the case $\ka_2\ge 5$.
We will integrate by parts with respect to $x_1$ in \eqref{eq:Kss'} and observe 
\Be \notag 
\begin{aligned} 
    &\partial_{x_1}\phi(x,t,y,z)= \partial_{x_1}\psi(x,t,y)-\partial_{x_1}\psi(x,t,z)\\
   &= y_1-z_1+ e_1^\intercal S P^\intercal (y'-z')
    = y_1-z_1-|Se_1| \inn{y'-z'}{\fe_2}.
\end{aligned}\Ee
In the present case $|Se_1|=\delta_l$ and   $(\fs,\fs')\in \cU_{\ka_1,\ka_2,\ka_3}$ with $\ka_2\ge 5$ and thus for $y\in \fs$, $z\in \fs'$
\[   |\inn{y-z}{\fe_2}| \ge 2^{\ka_2-1+\ka_1}\la^{-1} \delta_l^{-1} 2^{l\eps/d}, \qquad   |y_1-z_1| \le 2^{\ka_1+1}\la^{-1} 2^{l\eps/d};\]
hence 
\Be \label{eq:x1phi}
|\partial_{x_1}\phi(x,t,y,z)| \approx |Se_1||\inn{y'-z'}{\fe_2}|\approx 2^{\ka_1+\ka_2}\la^{-1}  2^{l\eps /d} .
\Ee
Note that $\sigma$ does not depend on $x_1$ and  
$\partial^N_{x_1}\phi=0$ for $N\ge 2$. After $N$-fold integration by parts with respect to  $x_1$ we get
$|\sK_{\fs,\fs'}(y,z) |\lc (2^{\ka_1+\ka_2} 2^{l\eps/d})^{-N}$. As above,  the 
asserted estimate \eqref{eq:ka2large} follows by Schur's test. \qed

\subsubsection*{Proof 
of \eqref{eq:ka2large} in the case $|Se_1|\le 2^{-l}$ and proof of  \eqref{eq:ka3large}}  Notice that in view of the small support of $\chi$  we have in the present case $\sK_{\fs, \fs'}^{\la,l}=0$ when  $2^l\la^{-1}\ge 1$, so the case $l=k$ is trivial. In what follows we assume $l\le k-1$;  it will be crucial that in this case 
$\sigma(x',y_1)$, $\sigma(x',z_1) $ have the same sign for $y\in \fs$ and $z\in \fs'$.

If $|Se_1|\le 2^{-l}$ we have $\delta_l=2^{-l} $ and for the proof of 
\eqref{eq:ka2large} we have also $\ka_3\le 10$ and 
we shall prove the pointwise estimate 
\Be \label{eq:ptwboundfirstcase} |\sK_{\fs,\fs'}(y,z)| \lc_N (2^{\ka_1+\ka_2} 2^{l\eps/d})^{-N} \Ee under the assumption  that $y\in \fs$, $z\in \fs'$ satisfy 
\Be \label{eq:assumpfirstcase}
\begin{aligned} &2^{\ka_1+\ka_2-1}\la^{-1} 2^l\le 2^{-l\eps/d} |\inn{y-z}{\fe_2} |\le 2^{\ka_1+\ka_2+2}\la^{-1} 2^l ,
\\
&2^{-l\eps/d} |(y-z)^\perp | \lc 2^{\ka_1+\ka_2+11} \la^{-1} 2^l, \quad |Se_1|\le 2^{-l} . 
\end{aligned}
\Ee 
here $(y-z)^\perp: =\sum_{i=3}^d \inn{y-z}{\fe_i}\fe_i $.

Moreover for \eqref{eq:ka3large} we shall  prove \Be \label{eq:ptwboundsecondcase} |\sK_{\fs,\fs'} (y,z)|\lc_N (2^{\ka_1+\ka_2+\ka_3} 2^{l\eps/d})^{-N} \Ee  under the assumption that $\ka_3\ge 5$ and that $y\in \fs$, $z\in \fs'$ satisfy 
\Be \label{eq:assumpsecondcase}
\begin{aligned} 
&2^{\ka_1+\ka_2+\ka_3-1}\la^{-1} 2^l\le 2^{-l\eps/d} |(y-z)^\perp  |\le 2^{\ka_1+\ka_2+\ka_3+2}\la^{-1} 2^l,
\\
&2^{-l\eps/d} |\inn{y-z}{\fe_2}| \le 2^{\ka_1+\ka_2+2}\delta_l^{-1} \la^{-1}.
\end{aligned}
\Ee 


We  use the directional derivative $\inn{\frac{y'-z'}{|y'-z'|}}{\partial_{x'}}$ in our integration by parts argument.
From \eqref{eq:mixhessian} we get (with $w^{\tau}$ as in \eqref{eq:wtau}) 
\begin{align}\label{eq:phix'expr}
    \partial_{x'}\phi(x,t,y,z) &= \int_0^1 \partial_{y} ^\intercal \partial_{x'}  \psi (x,t, w^\tau) (y-z)d\tau 
    \\ \notag= \int_0^1\Big[ &-g'(\tfrac{x'-{w^\tau}'}{t}) (y_1-z_1) +\frac{\sigma(x',w_1^\tau)}t g''(\tfrac{x'-{w^\tau}'} {t}) (y'-z') \\\notag &+ PSP^\intercal (y'-z') + PSe_1  (g'(\tfrac{x'-{w^\tau}'} {t})^\intercal(y'-z')) \Big]  \, d\tau.
\end{align} Take the scalar product with $\frac{y'-z'}{|y'-z'|}$ and use  that 
 $(y'-z')^\intercal PSP^\intercal (y'-z') =0$  to  get 
\Be\label{eq:lowerbd=dec}
\begin{aligned} 
\inn {\tfrac{y'-z'}{|y'-z'|}} {\partial_{x'}} \phi(x,t,y,z)\,=\,&\int_0^1 
\frac{\sigma(x',w_1^\tau) }{t} 
d\tau \,\frac{ (y'-z')^\intercal g''(0) (y'-z') }{|y'-z'|}
\\
&+R_1(x,t,y,z)+R_2(x,t,y,z)
\end{aligned}
\Ee
where 
\[ R_1(x,t,y,z)=\big ( \tfrac{y'-z'}{|y'-z'|})^\intercal \int_0^1 \frac{\sigma(x',w_1^\tau) }{t}
\big( g''(\tfrac{x'-{w^\tau}'} {t}) - g''(0)\big) d\tau (y'-z')
\] and 
\begin{align*} &R_2(x,t,y,z) = 
\\&\int_0^1\Big[  -\inn{ \tfrac{y'-z'}{|y'-z'|}}{g'(\tfrac{x'-{w^\tau}'}{t}) }(y_1-z_1) + 
\inn{ \tfrac{y'-z'}{|y'-z'|}}{PSe_1} (g'(\tfrac{x'-{w^\tau}'} {t})^\intercal(y'-z') )\Big] d\tau
\\
&= -\big(y_1-z_1- |Se_1| \inn{y'-z'}{\fe_2}\big)  \int_0^1 g'(\tfrac{x'-{w^\tau}'} {t})^\intercal \big ( \tfrac{y'-z'}{|y'-z'|} \big) d\tau.
\end{align*}
  
By the single-signedness of $\sigma$ we have  $|\int_0^1 t^{-1} \sigma(x',w_1^\tau) 
d\tau |\gc 2^{-l} $; here we use \eqref{eq:sigma-taudec}. 
Hence,   because of the positive definiteness of $g''(0)$ we see that the main term in  \eqref{eq:lowerbd=dec} satisfies the lower bound 
 \[\Big|\int_0^1 
\frac{\sigma(x',w_1^\tau) }{t} 
d\tau \,\frac{ (y'-z')^\intercal g''(0) (y'-z') }{|y'-z'|}\Big| \gc 2^{-l} |y'-z'|
\]
and we use 
\Be \notag
\begin{aligned}
&2^{-l}|y'-z'|\approx 2^{-l} |\inn{y-z}{\fe_2}| \approx 2^{l\eps/d}2^{\ka_1+\ka_2}  \la^{-1}\,\,\,\,\,\,\,\,\text{ if \eqref{eq:assumpfirstcase} holds},
\\
&2^{-l}|y'-z'|\approx  2^{-l} |(y-z)^\perp| \approx 2^{l\eps/d}2^{\ka_1+\ka_2+\ka_3} \la^{-1} \,\, \,\,\text{ if \eqref{eq:assumpsecondcase} holds}.
\end{aligned}
\Ee 

Since $\|g''(\tfrac{x'-{w^\tau}'} {t}) - g''(0)\|=O(\vr)$ we get 
\[|R_1(x,t,y,z) |\lc \vr 2^{-l} |y'-z'| 
\lc  
\begin{cases} \vr   2^{\ka_1+\ka_2}\la^{-1}  2^{l\eps/d}&\text{ if \eqref{eq:assumpfirstcase} holds}
\\ \vr 2^{\ka_1+\ka_2+\ka_3}\la^{-1} 2^{l\eps/d} &\text{ if \eqref{eq:assumpsecondcase} holds}.
\end{cases} 
\]

Finally \[  |R_2(x,t,y,z)| \lc \vr \big( |y_1-z_1|+|Se_1| |\inn{y'-z'}{\fe_2} |)\] 
and we have $|y_1-z_1|\lc 2^{\ka_1}\la^{-1}2^{l\eps/d}$  and thus clearly 
\[|R_2(x,t,y,z)|\lc \vr \big(|y_1-z_1|+2^{-l}|y'-z'| \big) \lc \vr 2^{\ka_1+\ka_2} \la^{-1}  2^{l\eps/d}  \text{ if \eqref{eq:assumpfirstcase} holds.} \] 
Moreover we  get this  when \eqref{eq:assumpsecondcase} holds and $|Se_1|\le 2^{-l}$. 

Now consider the case that  \eqref{eq:assumpsecondcase} holds and $|Se_1|\ge 2^{-l}$. Then \[|Se_1| |\inn{y'-z'}{\fe_2}| \lc 2^{\ka_1+\ka_2+2} \la^{-1}2^{l\eps/d} \] and thus we also get 
\[|R_2(x,t,y,z)|\lc \vr 2^{\ka_1+\ka_2} \la^{-1}  2^{l\eps/d}  \text{ if \eqref{eq:assumpsecondcase} holds.} \] 

Altogether, for $y\in \fs$, $z\in \fs'$, 
    \Be
|\inn {\tfrac{y'-z'}{|y'-z'|}} {\partial_{x'}} \phi(x,t,y,z)|\gc  2^{\ka_1+\ka_2}\la^{-1} 2^{l\eps/d} \text{ if \eqref{eq:assumpfirstcase} holds,}
\Ee and 
\Be |\inn {\tfrac{y'-z'}{|y'-z'|}} {\partial_{x'}} \phi(x,t,y,z)|\gc   2^{\ka_1+\ka_2+\ka_3}\la^{-1}2^{l\eps/d} \text{ if \eqref{eq:assumpsecondcase} holds.}
\Ee

We need corresponding upper bounds for the higher derivatives 
$\inn {\tfrac{y'-z'}{|y'-z'|}} {\partial_{x'}}$. First observe that
\Be\label{eq:sigmader}  \inn {\tfrac {y'-z'}{|y'-z'|} }{\partial_{x'}} \sigma(x',y_1)=  \inn {\tfrac {y'-z'}{|y'-z'|} }{PSe_1}. \Ee
Clearly this is $O(2^{-l})$ when $\delta_l=2^{-l}$, in particular under assumption \eqref{eq:assumpfirstcase}.  On the other hand, if $\delta_l>2^{-l}$ then we use that 
$PSe_1= |Se_1|\fe_2$ and if we now assume
\eqref{eq:assumpsecondcase} we have 
$|\inn{y'-z'}{PSe_1} |\le \delta_l |\inn{y-z}{\fe_2}|\le 2^{\ka_1+\ka_2+2}\la^{-1} 2^{l\eps/d} $ and $|y'-z'| \ge |(y-z)^\perp|\ge 2^{\ka_1+\ka_2+\ka_3-1} \la^{-1} 2^l 2^{l\eps/d}$; hence $\inn {\tfrac {y'-z'}{|y'-z'|} }{PSe_1}=O(2^{-l})$ and therefore 
$\inn {\tfrac {y'-z'}{|y'-z'|} }{\partial_{x'}} \sigma(x',y_1)=  O(2^{-l}).$
 Moreover, for the higher derivatives we have  $\inn {\tfrac {y'-z'}{|y'-z'|} }{\partial_{x'} }^N\sigma=0$ for $N\ge 2$.  This implies,  for all $N$,  
\Be \notag\big|\inn {\tfrac {y'-z'}{|y'-z'|} }{\partial_{x'} }^N [\eta_l(x,t,y,z)] \big|\lc_N 1.\Ee 

Differentiating in \eqref{eq:lowerbd=dec} 
and using these  estimates for $\sigma$ and $ \inn {\tfrac {y'-z'}{|y'-z'|} }{\partial_{x'}} \sigma$, also yields 
\Be\notag \begin{aligned} \big|  \inn {\tfrac {y'-z'}{|y'-z'|} }{\partial_{x'} }^N \phi(x,t,y,z)| &\lc 2^{-l} |y'-z'|+|y_1-z_1|+|Se_1| \inn{y'-z'}{\fe_2} | \\&\lc  
\begin{cases}2^{\ka_1+\ka_2}  \la^{-1} 2^{l\eps/d} &\text{ if \eqref{eq:assumpfirstcase} holds}\\
    2^{\ka_1+\ka_2+\ka_3}  \la^{-1} 2^{l\eps/d} &\text{ if \eqref{eq:assumpsecondcase} holds}
\end{cases}
\end{aligned}
\Ee 
An integration by parts yields 
\Be
\label{eq:Kss'mod}
\sK_{\fs,\fs'}(y,z)= \bbone_{\fs}(y) \bbone_{\fs'}(z)   \int e^{i\la\phi(x,t,y,z)} \frac{\cL^N\eta_l(x,t,y,z)}{ (-i\la)^N}\,dx dt \Ee
where
\[\cL g(x,t,y,z)= \biginn{\frac{y'-z'}{|y'-z'|} }{\partial_{x'}} \Big( \frac{g}{
\tfrac{y'-z'}{|y'-z'|} {\partial_{x'} \phi}} \Big);\] and we have 
\[|\cL^N [\eta_l(x,t,y,z)]|\lc \begin{cases} 
(2^{\ka_1+\ka_2}   2^{l\eps/d})^{-N}\la^N &\text{ if \eqref{eq:assumpfirstcase} holds}\\
(2^{\ka_1+\ka_2+\ka_3}   2^{l\eps/d})^{-N}\la^N &\text{ if \eqref{eq:assumpsecondcase} holds.}
\end{cases} 
\]
By \eqref{eq:Kss'mod} this leads to  the pointwise estimates \eqref{eq:ptwboundfirstcase} (under assumption \eqref{eq:assumpfirstcase}) and 
\eqref{eq:ptwboundsecondcase} (under assumption \eqref{eq:assumpsecondcase}).
By applying Schur's test we obtain the claimed bounds in both  cases.
\qed

\section{Open problems and further directions}

\subsection{$d=2$} The problem of nontrivial $L^p$ bounds for the Nevo-Thangavelu maximal operator when $d=2$ remains currently open even in the model case of the Heisenberg group $\bbH^1$.

\subsection{A restricted weak type endpoint bound} Theorem \ref{thm:rwt} establishes a restricted weak type $(\frac{d}{d-1}, \frac{d}{d-1})$ endpoint estimate for the local maximal operator, when $d\ge 3$. 
Does this endpoint  bound also hold for the global  operator? 
This is  the case when all $J_i$ are zero (\cf. \cite{Bourgain-CompteR}).

\subsection{$L^p$-improving estimates} One can ask whether the local operator $f\mapsto \sup_{1\le t\le 2} |f*\mu_t|$ maps $L^p$ to $L^q$ for some $q>p$; this would imply corresponding  sparse bounds for the global maximal operator (see 
 \cite{BagchiHaitRoncalThangavelu}). As a model case for the case  $m=1$ the precise $q$-range for such results  should depend on the rank of $J_1$ (and no $L^p$ improving takes place when $J_1=0$). For  $m\ge 2$ the dependence on the matrices $J_1,\dots, J_m$ could be quite  complicated. The case of Heisenberg type groups is covered in \cite{RoosSeegerSrivastava-imrn}.

\subsection{Restricted dilation sets} One can also consider maximal functions with restricted dilation sets. The  $L^p\to L^p$ estimates with Minkowski dimension type assumptions are  rather straightforward; one can combine  the methods of this paper with elementary arguments  in \cite{SeegerWaingerWright, RoosSeegerSrivastava-Studia}. In contrast the $L^p$-improving estimates  are  harder; for the Heisenberg groups $\bbH^n$, with $n\ge 2$, this problem was considered  in \cite{RoosSeegerSrivastava-Studia}. For general dilation sets there is a   large variety of possible type sets  (\cf. \cite[Thm.1.2]{RoosSeeger}), and much remains  open.

\subsection{Higher step groups}  It would be interesting to develop versions of our theorem which apply in the  general setting of stratified groups; here only the case of lacunary dilations is well understood  (see  e.g. \cite{Aswinetal}).

\subsection{Averages over tilted measures}  
The above problems can also be formulated for the case where the spherical measure $\mu$ 
is no longer supported in a subspace invariant under the automorphic dilations.
 The assumption of invariance under automorphic dilations  is crucial for the analysis   in the present paper but  it has been relaxed in  
 \cite{AndersonCladekPramanikSeeger, RoosSeegerSrivastava-imrn} which cover results on maximal functions associated with such tilted measures on Heisenberg or Heisenberg type groups.

\bibliographystyle{amsplain}
\providecommand{\bysame}{\leavevmode\hbox to3em{\hrulefill}\thinspace}
\providecommand{\MR}{\relax\ifhmode\unskip\space\fi MR }


\end{document}